\documentclass[11pt]{amsart}

\usepackage{hyperref}
\hypersetup{
   bookmarks=true,         
   unicode=false,          
   pdftoolbar=true,        
   pdfmenubar=true,        
   pdffitwindow=false,     
   pdfstartview={FitH},    
   pdftitle={My title},    
   pdfauthor={Author},     
   pdfsubject={Subject},   
   pdfcreator={Creator},   
   pdfproducer={Producer}, 
   pdfkeywords={keywords}, 
   pdfnewwindow=true,      
   colorlinks=true,       
   linkcolor=blue,          
   citecolor=blue,        
   filecolor=magenta,      
   urlcolor=MidnightBlue          
}

\usepackage[babel=true,kerning=true]{microtype}

\usepackage{amsthm}
\usepackage{graphicx, ifpdf, comment}
\usepackage{amsfonts,amsmath,amssymb,enumerate}
\usepackage{tikz}
\allowdisplaybreaks
\usetikzlibrary{decorations.markings}
\usetikzlibrary{plotmarks}
\usetikzlibrary{patterns}

\newcommand{\be}{\begin{equation*}}
\newcommand{\ee}{\end{equation*}}
\newcommand{\beq}{\begin{equation}}
\newcommand{\eeq}{\end{equation}}
\newcommand{\begincal}{\begin{eqnarray*}}
\newcommand{\fincal}{\end{eqnarray*}}

\newtheorem{thm}{Theorem}[section]
\newtheorem{lemma}[thm]{Lemma}
\newtheorem{cor}[thm]{Corollary}
\newtheorem{prop}[thm]{Proposition}
\newtheorem{defi}[thm]{Definition}

\newtheorem{rem}[thm]{Remark}
\newtheorem*{remark*}{Remark}

\newcommand{\eps}{{\varepsilon}}
\newcommand{\R}{{\mathbf R}}

\def\al{\alpha}

\def\eps{\varepsilon}

\def\Om{\Omega}

\def\p{\partial}

\newcommand{\mc}{\mathcal}
\newcommand{\mb}{\mathbf}

\newcommand{\dist}{\operatorname{dist}}

\newcommand{\cN}{\mathcal N}

\DeclareMathOperator{\vol}{Vol}
\DeclareMathOperator{\area}{Area}

\begin{document}

\title[Min-max minimal disks with free boundary in Riemannian manifolds]{Min-max minimal disks with free boundary in Riemannian manifolds}

\author{Longzhi Lin}
\address{Mathematics Department\\University of California - Santa Cruz, 1156 High Street\\ Santa Cruz, CA 95064\\USA}
\email{lzlin@ucsc.edu}

\author{Ao Sun}
\address{Department of Mathematics\\Massachusetts Institute of Technology\\ 77 Massachusetts Avenue\\Cambridge, MA 02139\\USA}
\email{aosun@mit.edu}

\author{Xin Zhou}
\address{Department of Mathematics\\University of California - Santa Barbara\\ Santa Barbara, California 93106\\USA}
\email{zhou@math.ucsb.edu}

\date{\today}


\maketitle
\begin{abstract}
In this paper, we establish a min-max theory for constructing minimal disks with free boundary in any closed Riemannian manifold. The main result is an effective version of the partial Morse theory for minimal disks with free boundary established by Fraser. Our theory also includes as a special case the min-max theory for Plateau problem of minimal disks, which can be used to generalize the famous work by Morse-Thompkins and Shiffman on minimal surfaces in $\mathbf{R}^n$ to the Riemannian setting.\vskip 1.5mm

More precisely, we generalize the min-max construction of minimal surfaces using harmonic replacement introduced by Colding and Minicozzi to the free boundary setting. As a key ingredient to this construction, we show an energy convexity for weakly harmonic maps with mixed Dirichlet and free boundaries from the half unit $2$-disk in $\mathbf{R}^2$ into any closed Riemannian manifold, which in particular yields the uniqueness of such weakly harmonic maps. This is a free boundary analogue of the energy convexity and uniqueness for weakly harmonic maps with Dirichlet boundary on the unit $2$-disk proved by Colding and Minicozzi. 
\end{abstract}

\tableofcontents

\setcounter{section}{-1}

\section{Introduction}
Given a closed Riemannian manifold $\mathcal{N}^n$ (that is isometrically embedded in $\mathbf{R}^N$) and an embedded submanifold $\Gamma$ of co-dimension $l\geq 1$, a map from the unit $2$-disk $D\subset \mathbf{R}^2$ into $\mathcal{N}$ with boundary lying on $\Gamma$ is said to be a {\em minimal disk with free boundary} if it minimizes area up to the first order among all such maps. One physical model of such surfaces is the soap film whose boundary is constrained (but allowed to move freely) on the boundary of some smooth domain in $\R^3$. When a soap film achieves the equilibrium state, it will minimize the area up to first order. Geometrically, the stationary soap film would have vanishing mean curvature in the interior and meet the boundary of the given domain orthogonally. Therefore, the orthogonality condition at the boundary is called the {\em free boundary} condition. \vskip 1.5mm

After earlier works of Gergonne in 1816 and H. Schwarz in 1890, Courant first studied systematically the free boundary problems for minimal surfaces in a series of seminal papers (see \cite[Chapter VI]{Courant} and also \cite{Courant-Davids40}).  In particular, together with Davids, they proved that given an embedded closed surface $S$ in $\R^3$ other than the sphere, there exists a minimal disk $\Sigma$ with free boundary on $S$ under certain linking conditions; (see \cite[p.213-218]{Courant}). Since then, there have been immense research activities on this topic. To remove the topological assumption on the constraint surface $S$, Smyth \cite{Smyth84} showed that if $S$ is the boundary of a tetrahedron (which is \emph{non-smooth}), then there must exist exactly three minimal disks embedded inside the tetrahedron satisfying the free boundary condition. When $S$ is a smooth topological two-sphere, Struwe \cite{Struwe84} used the mountain pass lemma to establish the existence of at least one \emph{unstable} minimal disk with free boundary lying on $S$. In higher dimensions and co-dimensions (for $(\mathcal{N}, \Gamma)$), Ye \cite{Ye91a} obtained the existence of an area minimizing disk with free boundary when the kernel of $\pi_1(\Gamma)\to \pi_1(\mathcal{N})$ is non-trivial. About 20 years ago, Fraser \cite{Fraser00} developed a partial Morse theory for finding minimal disks with free boundary in any co-dimensions using the perturbed energy approach by Sacks-Uhlenbeck \cite{Sacks-Uhlenbeck81}, and proved the existence of solutions with bounded Morse index when the relative homotopy group $\pi_k(\mathcal{N}, \Gamma)$ is non-trivial for some $k\in\mathbb{N}$. Recently, Fraser-Schoen \cite{Fraser-Schoen11, Fraser-Schoen16} established deep relations between the minimal surfaces with free boundary in round balls and the extremal eigenvalue problems. \vskip 1.5mm

In this paper, we develop a direct variational theory for constructing minimal disks with free boundary for any pair $(\mathcal{N}, \Gamma)$ using min-max method. In particular, given a $k$-parameter family of mappings from the unit disk $D$ into $\mathcal{N}$, such that $\partial D$ are mapped into $\Gamma$, one can associate with it a min-max value analogous to the classical Morse theory. We prove that if the min-max value is non-trivial, then there exists a minimal disk together with (possibly empty) several minimal spheres (possibly with a puncture), which is usually called a \textit{bubble tree} and whose areas sum up to be the min-max value (the so-called \textit{energy identity}). Moreover, we prove that every approximate sequence of maps converges to a bubble tree such that the energy identity holds true. By reflecting on this strong convergence property, our results can be viewed as an effective version of Fraser \cite{Fraser00}; (see more discussion in Remark \ref{R:remark for main1} (2)). \vskip 1.5mm

Moreover, our theory can be used to generalize the famous work of Morse-Thompkins \cite{Morse-Tompkins39} and Shiffman \cite{Shiffman39} on minimal surfaces in $\R^n$ to the Riemannian setting (Theorem \ref{T:main2}). Particularly, take $\Gamma$ to be a Jordan curve in $\mathcal{N}$, and assume that $\Gamma$ bounds two different strictly minimizing minimal disks $\bar v_i:D\to \cN, i=1,2$ so that $\bar v_i\vert_{\partial D}: \partial D \to \Gamma$ is a monotone parametrization, our theory produces finitely many harmonic disks $u_k$'s (and minimal spheres), such that when restricted to $\partial D$, only one of which will have degree $1$ and all others will have degree $0$. If additionally one has that $u_k\vert_{\partial D}: \partial D \to \Gamma$ are monotone parametrizations, then in the special case when $\mathcal{N}$ has non-positive curvature (so that there exists no punctured minimal spheres), our theory produces a third \textit{non-minimizing} minimal disk, and therefore it provides a direct generalization of the work of Morse-Tompkins \cite{Morse-Tompkins39} and Shiffman \cite{Shiffman39}; see Section \ref{S:Necessary modifications for the proof of fixed boundary min-max} for more details.\vskip 1.5mm

Another novelty of this paper is reflected by our constructive method. The \textit{Schwarz alternating method} introduced by H. Schwarz goes back to the late 1860's, and later it was generalized to an iterative method for finding the solution of an elliptic PDE on a domain which is the union of two overlapping subdomains. In \cite{CM08}, Colding and Minicozzi adapted this method and used the \textit{harmonic replacement} to construct min-max minimal surfaces. During this repeated replacement procedure, at each step one replaces a map $u$ by a map $\tilde{u}$ that coincides with $u$ outside a disk and inside the disk is equal to an energy-minimizing map with the same boundary values as $u$. A key ingredient to this construction is a version of energy convexity for weakly harmonic maps with Dirichlet boundary and small energy on the unit $2$-disk $D$, which also yields the (quantitative) uniqueness for such weakly harmonic maps; see also Lamm and the first author \cite{LL13}. In this paper, we generalize this min-max construction of Colding-Minicozzi using harmonic replacement to construct minimal disks with free boundary in any closed Riemannian manifold. To this end, we will show an energy convexity for weakly harmonic maps with mixed Dirichlet and free boundaries from the half unit $2$-disk into any closed Riemannian manifold (Theorem \ref{T:convexity2}), which is the free boundary analogue of Colding-Minicozzi's energy convexity and uniqueness for weakly harmonic maps. We shall remark that a priori it is not at all clear if such energy convexity should hold true due to the complication of the free boundary component of the map. The key to this is an $\epsilon$-regularity (gradient estimate) for weakly harmonic maps with mixed Dirichlet and free boundaries on the half unit $2$-disk (Theorem \ref{e-reg}).\vskip 1.5mm

Now we proceed to present the precise mathematical statements of our main results. To make the presentation simpler, we will focus on 1-parameter min-max constructions, though our results extend in a straightforward manner to $k$-parameters. We will use $[0, 1]$ as the parameter space. Consider the total variational space:
\[ \Omega =\left\{\sigma: D\times[0,1]\to \cN, \text{ such that: } \begin{array}{l} \sigma: [0, 1] \to C^0(\overline{D}, \cN) \cap W^{1,2}(D, \cN) \text{ is continuous, }\\
                             \text{and } \sigma(\cdot, t)(\partial D)\subset \Gamma, \,\forall \,t\in [0,1],\\
                             \text{and } \sigma(\cdot, 0), \sigma(\cdot, 1) \text{ are constant maps}.            
                    \end{array}\right\}. \]
Each $\beta\in\Omega$ will be called a {\em sweepout}. Given a map $\beta \in\Omega$, we define $\Omega_\beta$ to be the homotopy class of $\beta$ in $\Omega$.
\vskip 1.5mm
Here and in the following, we denote $E(\cdot)$ and $\area(\cdot)$ as the Dirichlet energy and area functionals on $C^0(\overline{D}, \cN) \cap W^{1,2}(D, \cN)$. Associated to each homotopy class $\Omega_\beta$, there is a min-max value, also called {\em width} of $\Omega_\beta$:
\begin{equation}\label{E:area width1}
    W=W(\Om_\beta):=\inf_{\gamma\in\Omega_\beta} \max_{s\in [0, 1]}\area(\gamma(\cdot, s)).
\end{equation}

As the first main result, we establish a direct variational construction for minimal surfaces associated with this critical value $W$.
\begin{thm}
\label{T:main1}
Given $\beta\in\Omega$ with $W = W(\Om_\beta)>0$, there is a sequence of sweepouts $\gamma^j\in\Omega_\beta$ with $\max_{s\in[0,1]}E(\gamma^j(\cdot, s))\to W$, such that for any given $\epsilon>0$, there exist $\bar j$ and $\delta>0$ so that if $j>\bar j$ and 
\[\area(\gamma^j(\cdot,s))>W-\delta,\]
then there exist finitely many harmonic maps $u_k: D\to \cN$ with $u_k(\partial D)\subset \Gamma$ and finitely many (possibly empty) harmonic maps $\hat{u}_l: \mathbf{S}^2\to \cN$, so that
\[d_V(\gamma^j(\cdot,s), \cup_{k}\{u_k\}\cup_{l}\{\hat{u}_l\}) \to 0, \text{ as } j\to \infty. \]
Here we have identified each map $u_i$ with the varifold associated to the map, and $d_V$ denotes the varifold distance. Moreover we have the energy identity
\[ \sum_k \area(u_k)+ \sum_l \area(\hat{u}_l)=W. \]
\end{thm}

\begin{rem}\label{R:remark for main1}
We have the following remarks ready.
\begin{enumerate}
\item By the work of Sacks-Uhlenbeck \cite{Sacks-Uhlenbeck81} and Fraser \cite{Fraser00}, we know that these harmonic maps are conformal and hence parametrize minimal disks with free boundary on $\Gamma$ or minimal spheres.
\vskip 2mm
\item Our result is an effective version of Fraser's result \cite{Fraser00} in the sense that we obtain a strong convergence property. In particular, a sequence of maps $\{\gamma^{j_k}(\cdot, s_k)\}$ is usually called a min-max sequence, if $$\lim_{k\to\infty}\area(\gamma^{j_k}(\cdot, s_k))=W\,.$$ We prove that every min-max sequence will sub-converge to a set of minimal disks with free boundary and possibly some minimal spheres in the varifold sense. The essential ingredient is to prove that all the min-max sequences converge to a bubble tree limit and the energy identity holds.  Note that the energy identity has caught a lot of attention from mathematicians in conformally invariant variational problems, see e.g. \cite{Jost91, Parker-Wolfson93, Ding-Tian95, Parker96, Qing-Tian97, Chen-Tian99, Ding-Li-Liu06}; for quantification results for harmonic maps with free boundary, see e.g.  \cite{LP17, JLZ16} . A similar strong convergence property was first proven by Colding-Minicozzi in \cite{CM08} for the min-max construction of minimal spheres, and it played an essential role in their proof of the finite time extinction for certain 3-dimensional Ricci flow.  Similar property was also obtained by the last author for the min-max construction of closed minimal surfaces of higher genus \cite{Zhou10, Zhou17}, and by Rivi\'ere for min-max construction of closed minimal surfaces via viscosity method \cite{Riviere17}. To the authors' knowledge, our work is the first occasion to obtain such a strong property in the context of free boundary problems.
\vskip 2mm
\item  As a special case of our result, one can take $(\mc N, \Gamma)$ to be a compact manifold with convex boundary $(M, \partial M)$. Using the convex boundary as barriers, our theory applies in this case and the resulting minimal disks with free boundary on $\partial M$ and minimal spheres will all lie inside $M$.
\end{enumerate}
\end{rem}

Our main result has almost a direct corollary for a min-max construction of minimal disks with fixed boundary. In particular, we now assume $\Gamma$ to be a Jordan curve in $\mathcal{N}$. Suppose $\bar v_0:D\to \cN$ and $\bar v_1:D\to \cN$ are two area minimizing minimal disks (conformal and harmonic maps), where $\bar v_i\vert_{\partial D}: \partial D \to \Gamma$ is a monotone parametrization, $i=1,2$. The total variational space for the fixed boundary problem will be
\[ \Omega_f =\left\{\sigma: D\times[0,1]\to \cN, \text{ such that: } \begin{array}{l} \sigma: [0, 1] \to C^0(\overline{D}, \cN) \cap W^{1,2}(D, \cN) \text{ is continuous, }\\
                             \text{and } \sigma(\cdot, t)(\partial D)\subset \Gamma,\\                             \text{and } \sigma(\cdot, 0)=\bar v_0, \sigma(\cdot, 1)=\bar v_1.            
                    \end{array}\right\}. \]
Given a map $\beta \in\Omega_f$, we define $\Omega_\beta$ to be the homotopy class of $\beta$ in $\Omega_f$. 
The width $W$ associated with $\Omega_\beta$ can be defined in the same way, namely
\begin{equation}
    W=W(\Om_\beta):=\inf_{\gamma\in\Omega_\beta} \max_{s\in[0, 1]}\area(\gamma(\cdot, s)).
\end{equation}
Then $W\geq \max(\area(\bar v_0), \area(\bar v_1))>0$. The next result is a slight variant of Theorem \ref{T:main1}.

\begin{thm}
\label{T:main2}
Given $\beta\in\Omega_f$ with $W = W(\Om_\beta)>\max(\area(\bar v_0), \area(\bar v_1))$, there is a sequence of sweepouts $\gamma^j\in\Omega_\beta$ with $\max_{s\in[0,1]}E(\gamma^j(\cdot, s))\to W$, such that for any given $\epsilon>0$, there exist $\bar j$ and $\delta>0$ so that if $j>\bar j$ and 
\[\area(\gamma^j(\cdot,s))>W-\delta,\]
then there exist finitely many harmonic disks $u_k: D\to \cN$ with $u_k(\partial D)\subset \Gamma$ and finitely many (possibly empty) harmonic spheres $\hat{u}_l: \mathbf{S}^2\to \cN$ with
\[d_V(\gamma^j(\cdot,s),\cup_k\{u_k\}\cup_{l}\{\hat{u}_l\}) \to 0, \text{ as } j\to \infty. \]
Moreover, when restricted as maps from $\partial D$ to $\Gamma$, only one map among $u_k$'s has degree $1$, whereas all others have degree $0$. We also have the energy identity $\sum_k \area (u_k)+ \sum_l \area(\hat{u}_l)=W$.
\end{thm}
\begin{remark*}
We postpone the discussions of this result until Section \ref{S:Necessary modifications for the proof of fixed boundary min-max}.
\end{remark*}

\vspace{1em}
We also want to mention the min-max theory for constructing minimal submanifolds with free boundary using geometric measure theory. In 1960s, Almgren \cite{Almgren62} initiated a program to develop a Morse theory for minimal submanifolds (with or without free boundary), and he obtained the existence of an integral varifold which is stationary with free boundary in the sense of first variation in any dimensions and co-dimensions in \cite{Almgren65}; (see more details in \cite{LZ16}). Along this direction higher regularity was established for hypersurfaces later. In particular, Gr\"{u}ter and Jost \cite{Gruter-Jost86a} proved the existence of an unstable embedded minimal disk inside any compact \emph{convex} domain in $\R^3$. Later, Jost in \cite[Theorem 4.1]{Jost86} generalized their work to any compact three-manifold which is diffeomorphic to a ball with mean convex boundary. Higher dimensional results were developed very recently by De Lellis-Ramic \cite{Delellis-Ramic16} (for both free and fixed boundary problems in convex manifolds; see also \cite{Montezuma18}), and by Li and the last author \cite{LZ16} (for the free boundary problem in any compact manifolds with boundary). We refer to \cite{MN-survey} for other recent developments of the min-max theory using geometric measure theory.

\subsection*{Sketch of main ideas}  Here we provide a brief summary of our main new ideas. Though the main scheme follows the approach laid out by Colding-Minicozzi \cite{CM08} for minimal spheres (see also \cite{Zhou10, Zhou17} for minimal surfaces with higher genus), the presence of free boundary in our setting brings in several main new obstacles. 
\vskip 1.5mm

In the analytic aspect, there are two main ingredients that we have to establish for weakly harmonic maps with (partial) free boundary. The first ingredient is a version of energy convexity which says that the energy functional is strictly convex near a weakly harmonic map with mixed Dirichlet and free boundaries on the half 2-disk. Unlike the proof in \cite{CM08} where Colding-Minicozzi used the moving frame method developed by H\'elein \cite{He02} in order to get a Hardy type estimate for weakly harmonic maps with Dirichlet boundary and small energy on the 2-disk $D$, we first use {U}hlenbeck-{R}ivi\`ere decomposition method developed by {R}ivi\`ere \cite{Ri07} to get a refined $\epsilon$-regularity (gradient estimate) for weakly harmonic maps with mixed Dirichlet and free boundaries on the half 2-disk (Theorem \ref{e-reg}), and then appeal to the first order Hardy inequality (Lemma \ref{hardyinequ}), c.f. Lamm and the first author \cite{LL13} where the energy density $|\nabla u|^2$ is estimated in the local Hardy space $h^1(D)$. The key geometric observation in the proof of the energy convexity lies in two orthogonality conditions, one observed by Colding-Minicozzi \cite{CM08}, and the other by the third author \cite{Zhou16}. Very recently, this idea also permits the first author and Laurain to obtain an energy convexity and uniqueness for \textit{weakly intrinsic bi-harmonic maps} defined on the unit $4$-ball with small bi-energy, which in particular yields a version of uniqueness for weakly harmonic maps in dimension $4$, see \cite{LauL18}. The other ingredient is the uniform continuity estimate up to the free boundary for weakly harmonic maps. Whereas the continuity up to the Dirichlet boundary was proven by Qing \cite{Qing93}, we prove the uniform continuity for weakly harmonic maps with mixed Dirichlet and free boundaries by a careful covering argument using our previous gradient estimate. We think that the results and techniques in both of the two ingredients have independent interests for the free boundary problem of harmonic maps.
\vskip 1.5mm

The variational construction mainly consists of two parts. The first part is a two-stage tightening process following closely the approach of Colding-Minicozzi \cite{CM08} (while this can be made as a $2^k$-stage process when the parameter space is $k$-dimensional). The main challenges here include the $W^{1,2}\cap C^0$-continuity for (partial) free boundary harmonic replacements, and an interpolation construction to prove energy improvement inequalities. In fact, the key idea in both places is to construct comparison maps and then use the energy minimality conditions. One natural candidate of comparison maps is the linear interpolation in $\mb R^N $of two maps $u_i: D\to \mc N$, $i=1, 2$. Though such map will go outside of $\mc N$, in \cite{CM08} Colding-Minicozzi used the nearest point projection to pull it back to $\mc N$. However the projection does not necessarily map the  image of $\partial D$ into $\Gamma$. Therefore we have to develop new methods. To prove the $W^{1, 2}$-continuity, we find a way to reduce the problem to an interpolation between two curves; and in the second place, we choose to do interpolation in Fermi coordinates. Based on the two new results, the tightening process then can be carried through in the free boundary setting analogously as in \cite{CM08}.
\vskip 1.5mm

The second part is a bubbling convergence procedure for almost harmonic maps with free boundary. The bubbling convergence for almost harmonic maps on spheres was developed by Colding-Minicozzi \cite[Appendix B]{CM08}, and for free boundary $\alpha$-harmonic maps it has been investigated systematically by Fraser \cite{Fraser00}. Our result can be viewed as a combination of the two results. Among several things, the most novel observation in this part is the asymptotic analysis for harmonic maps with free boundary defined on a long half cylinder (which is conformally equivalent to a thin half annulus). In particular, we prove that the angular energy is much smaller than the total energy of this map. Since boundary terms will appear in the integration by parts argument, we have to use our gradient estimate together with a delicate doubling argument to take care of the boundary terms, so as to carry out Colding-Minicozzi's method (which works on cylinders). 

\subsection*{Layout of the paper} 
The paper is organized as follows. In Section \ref{S:Notations}, we fix some notations. In Section \ref{S:Energy convexity for free boundary weakly harmonic maps}, we prove the energy convexity for weakly harmonic maps defined on the half disk with mixed Dirichlet and free boundaries. In Section \ref{S:Existence and regularity of free boundary harmonic replacement}, we prove that a weakly harmonic map with mixed Dirichlet and free boundaries on the half disk with continuous partial Dirichlet boundary is also continuous on the whole free boundary including the corner points; and we also present the proof on the existence of a weakly harmonic map with free boundary and with prescribed partial Dirichlet boundary (that we call \textit{partial free boundary harmonic replacement}). In Section \ref{S:Overview of the variational approach}, we outline the main results needed to establish the min-max theory in the free boundary setting. In Section \ref{S:Conformal parametrization}, we prove that the min-max values for the area and energy functionals are the same by using conformal reparametrizations. In Section \ref{S:Construction of the tightening process}, we carry out the construction of the tightening process; in particular, we show how to use a two-stage harmonic replacement procedure to make the sweepout as tight as possible. In Section \ref{S:Compactness of maximal slices}, we prove that any min-max sequence of maps will converge to a bubble tree consisting of harmonic disks and harmonic spheres; here we show that the bubbling convergence satisfies the energy identity, or equivalently, the total energy of the bubble tree is the same as the min-max value. In Section \ref{S:Necessary modifications for the proof of fixed boundary min-max}, we point out necessary changes needed to adapt our theory to the fixed boundary min-max problem (Theorem \ref{T:main2}).

\vspace{1em}
{\bf Acknowledgements}: The authors would like to thank Professor William Minicozzi for encouragements and helpful comments, and also Professor Rugang Ye for helpful discussions.

\vspace{1em}
{\bf Added in proof}: We were informed that Laurain and Petrides have recently proved a similar result, and Theorem \ref{T:convexity2} was announced in a seminar proceedings paper \cite{Lau17} but without a proof.

\section{Notations}
\label{S:Notations}
We first fix some notations. 
\begin{itemize}
\item $\mb R^2$ denotes the Euclidean two plane where $(x, y)$ (or $(x_1, x_2)$), $(r, \theta)$ are the Cartesian and polar coordinates respectively.
\item $\mb H^2=\{(x, y)\in\mb R^2: y> 0\}$ denotes the upper half plane.
\item $D_r = D_r(0) = \{ (x,y)\in \mb{R}^2: x^2+y^2 < r^2\}$ denotes the disk of radius $r$ centered at origin.
\item $D^+_s$ denotes the upper half disk with radius $s$, i.e. $D^+_s=D_s\cap \mc H^2=\{(r,\theta):0\leq r < s, 0< \theta < \pi\}$.
\item For simplicity, we sometimes write $D^+ = D^+_1$. 
\end{itemize}
Then we denote $\partial D_s^+=\partial^C_s \cup \partial^A_s$, which is the union of the chord (diameter) and the arc of the upper-semi circle, i.e.
\[ \partial^C_s=\{(r,\theta): 0\leq r \leq s,\,\, \theta=0\,\, \text{or}\,\, \pi\} \quad \text{and}\quad \partial^A_s=\{(r,\theta): r\equiv s, \,\,0\leq \theta \leq \pi\}\,. \]
 Similarly, we will denote 
\[ \partial^C = \partial^C_1 \quad \text{and} \quad  \partial^A = \partial^A_1\,.\]

We write $u: (D, \partial D)  \to (\mathcal{N}, \Gamma)$ if $u$ is a map from $D$ to $\mathcal{N}$ and $u(\partial D)\subset \Gamma$. Similarly, we write $u: (D^+_s, \partial_s^C)  \to (\mathcal{N}, \Gamma)$ if $u$ is a map from $D^+_s$ to $\mathcal{N}$ and $u(\partial_s^C)\subset \Gamma$.
\vskip 2mm

Given $u: D\to \mathcal{N}$ or $u: D^+_s\to \mathcal{N}$, the Dirichlet energy is defined as
\[ E(u):= \int_{D \text{ or } D_s^+} |\nabla u|^2\,dxdy\,. \]

The Euler-Lagrange equation for the Dirichlet energy is the so called {\em harmonic map equation}, which is a quasi-linear system defined by (see e.g. \cite{He02}): 
\begin{equation}
\label{E:harmonic map equation}
-\Delta u = A(u)(\nabla u, \nabla u),
\end{equation}
where $A(u)$ is the second fundamental form of the embedding $\mathcal{N}\hookrightarrow \mathbf{R}^N$.

\begin{defi}
A $W^{1, 2}$-map $u: (D, \partial D)  \to (\mathcal{N}, \Gamma)$ is called a {\em weakly harmonic map with free boundary} if $u$ satisfies the harmonic map equation weakly in $D$ and
\[ \frac{\partial u}{\partial r}\perp \Gamma \text{ along } \partial D.\]

Similarly a $W^{1, 2}$-map  $u: (D^+_s, \partial_s^C)  \to (\mathcal{N}, \Gamma)$ is called a {\em weakly harmonic map with partial free boundary} if $u$ satisfies the harmonic map equation weakly in  $D^+_s$ and 
\[ \frac{\partial u}{\partial y}\perp \Gamma \text{ along } \partial_s^C. \]
\end{defi}

\section{Energy convexity for weakly harmonic maps with partial free boundary}
\label{S:Energy convexity for free boundary weakly harmonic maps}

In this section, we present the first main result, that is, the energy convexity and uniqueness for weakly harmonic maps with mixed Dirichlet and free boundaries. This is not only one of the key ingredients of our min-max existence theory for minimal disks with free boundary, but it also has its own interest from the point of view of PDE's and calculus of variations.
\vskip 1.5mm

We first recall Colding-Minicozzi's energy convexity for weakly harmonic maps defined on the $2$-disk $D$ with Dirichlet boundary and small energy, see also \cite{LL13}.
\begin{thm}(\cite[Theorem 3.1]{CM08})
\label{T:convexity1}
There exists a constant $\varepsilon_0>0$ depending only on $\cN$ such that if $u, v \in W^{1,2}(D, \cN)$, $u = v$ on $\partial D$, and $u$ is weakly harmonic with $E(u) \leq \varepsilon_0$, then
\begin{equation}
\frac{1}{2}\int_{D}\vert\nabla v-\nabla u\vert^2\,dx \leq \int_{D}\vert\nabla v\vert^2\,dx-\int_{D}\vert \nabla u\vert^2\,dx\,.
\end{equation}
\end{thm}

In this section, we prove a free boundary analogue of Theorem \ref{T:convexity1}. We will abuse notation and still denote the energy threshold as $\varepsilon_0$. More precisely, we prove

\begin{thm}[Energy convexity for weakly harmonic maps with mixed Dirichlet and free boundaries]
\label{T:convexity2}
There exists a constant $\varepsilon_0>0$ depending only on $\mathcal{N}$ and $\Gamma$ such that if $u, v \in W^{1,2}(D^+, \mathcal{N})$ with $u|_{\partial^A}=v|_{\partial^A}$, $u|_{\partial^C} \subset \Gamma$, $v|_{\partial^C} \subset \Gamma$, $E(u) \leq \varepsilon_0$, and $u$ is a weakly harmonic map with partial free boundary, then we have the energy convexity
\begin{equation}\label{Convexity1}
\frac{1}{2}\int_{D^+} |\nabla v-\nabla u|^2\,dx\, \leq\, \int_{D^+} |\nabla v|^2 \,dx\,-\int_{D^+}|\nabla u|^2\,dx\,\,.
\end{equation}
\end{thm}

An immediate corollary of Theorem \ref{T:convexity2} is the uniqueness of weakly harmonic maps with mixed Dirichlet and free boundaries on $D^+$.

\begin{cor}
\label{C:uniqueness of small energy free boundary harmonic maps}
There exists $\varepsilon_0>0$ depending only on $\mathcal{N}$ and $\Gamma$ such that the following holds: for any two weakly harmonic maps $u, v \in W^{1,2}(D^+, \mathcal{N})$ with Dirichlet boundary $u|_{\partial^A}=v|_{\partial^A}$ and free boundaries $u|_{\partial^C} \subset \Gamma, v|_{\partial^C} \subset \Gamma$, if their energies satisfy $E(u) \leq \varepsilon_0$ and $E(v) \leq \varepsilon_0$, then we have $u \equiv v$ in $D^+$.
\end{cor}

In order to prove Theorem \ref{T:convexity2}, we will use the following first order Hardy's inequality and appeal to the refined $\epsilon$-regularity for weakly harmonic maps with mixed Dirichlet and free boundaries (Theorem \ref{e-reg}).

\begin{lemma}[Hardy's inequality] \label{hardyinequ} Let $u, v$ be in $W^{1,2}(D^+, \mathbf{R}^{N})$ with $u|_{\partial^A}=v|_{\partial^A}$, then we have
\begin{equation}\label{hardy}
\int_{D^+} |v-u|^2\cdot \frac{1}{(1-|x|)^2} \,dx \,\leq \, 4 \int_{D^+} |\nabla (v-u)|^2\,dx\,.
\end{equation}
\end{lemma}
\begin{proof}
First we extend $u$ and $v$ across $\partial^C$ by reflection (setting $\tilde{u}(x_1, x_2) = u(-x_1, x_2)$ and $\tilde{v}(x_1, x_2) = v(-x_1, x_2)$ for $x_1\leq 0$) to get $\tilde{v}-\tilde{u} \in W^{1,2}_0(D)$. Approximate $\tilde{v}-\tilde{u}$ in $W^{1,2}$ by a sequence of smooth functions with compact support $\tilde{w}_i \in C^\infty_c(D)$ and let $w_i$ be the restriction of $\tilde{w}_i$ on $D^+$. Now for each $w_i$ we have
\begin{align}
&\int_{D^+} \frac{\vert w_i \vert ^2}{(1-|x|)^2} \,dx \,=\, \int_0^\pi \int_0^1 \frac{\vert w_i \vert^2 r}{(1-r)^2} dr d\theta\notag\\
=&\, \int_0^\pi\left [ \frac{\vert w_i \vert^2 r}{1-r}\Big\vert_0^1 - \int_0^1 \left( \frac{\vert w_i \vert^2 }{1-r} +  \frac{2w_i \cdot (w_i)_r r}{1-r} \right)dr\right]d\theta\notag\\
\leq& \,2 \left(\int_0^\pi\  \int_0^1 \frac{\vert w_i \vert^2 r}{(1-r)^2} dr d\theta \right)^{1/2}\left(\int_0^\pi\  \int_0^1 \vert(w_i)_r\vert^2 r dr d\theta \right)^{1/2}\notag\\
\leq&\, 2\left(\int_{D^+} \frac{\vert w_i \vert^2}{(1-|x|)^2} \,dx\right)^{1/2}\left(\int_{D^+} |\nabla w_i|^2\,dx \right)^{1/2}\,,
\end{align}
which yields
\begin{equation}\label{hardy2}
\int_{D^+} \frac{\vert w_i \vert^2}{(1-|x|)^2} \,dx \,\leq \, 4 \int_{D^+} |\nabla w_i|^2\,dx\,.
\end{equation}
Now \eqref{hardy} follows from \eqref{hardy2} and the Fatou's lemma.
\end{proof}

The next result is a refined $\epsilon$-regularity for weakly harmonic maps on $D^+$ with Dirichlet and free boundaries, c.f. \cite{Scheven06, LP17, JLZ16}. This $\epsilon$-regularity is crucial to the proof of Theorem \ref{T:convexity2}. We shall remark that such $\epsilon$-regularity is well-known for weakly harmonic maps defined on the 2-disk $D$ with Dirichlet boundary, see e.g. Qing \cite[Lemma 4]{Qing93}.

\begin{thm}
\label{e-reg}
There exists a constant $\varepsilon_0>0$ depending only on $\mathcal{N}$ and $\Gamma$ such that if $u \in W^{1,2}(D^+, \mathcal{N})$ is a weakly harmonic map with mixed Dirichlet boundary on $\partial^A$ and free boundary $u|_{\partial^C} \subset \Gamma$ and $E(u)\leq\varepsilon_0$, then for any $x\in D^+ \cup \left(\partial^C\right)^\circ$ we have
\begin{equation}\label{ptest}
\vert \nabla u \vert(x) \leq \frac{C\sqrt{\varepsilon_0}}{1-|x|}\,,
\end{equation}
for some constant $C>0$ that only depends on $\mathcal{N}$ and $\Gamma$.
\end{thm}

\begin{proof} By the reflection across the free boundary $\Gamma$ constructed by Scheven in \cite{Scheven06}, $u\in W^{1,2}(D^+)$ can be extended to $\tilde{u} \in W^{1,2}(D)$ such that $\tilde{u}$ weakly solves in $D$ the system of equations:
\begin{equation}\label{eq01}
\text{div} (Q \nabla \tilde{u}) \,=\,\omega\cdot Q \nabla \tilde{u}
\end{equation}
for some $\omega=(\omega^i_j)_{1\leq i,j\leq n}\in L^2(D, so(n)\otimes\wedge^1\mathbf{R}^2)$ (i.e., $\omega_i^j = - \omega_j^i$) and $Q\in W^{1,2}\cap L^\infty (D, GL(n,\mathbf{R}))$ such that 
\begin{equation}\label{eq11}
\vert \omega \vert \leq C |\nabla \tilde{u}|  \quad \text{a.e. in } D\quad \text{and}\quad  \Vert Q \Vert_{L^\infty(D)} +  \Vert Q^{-1} \Vert_{L^\infty(D)} \leq C\,,
\end{equation}
where $C>0$ is a constant depending only on $\mathcal{N}$ and $\Gamma$, see Jost-Liu-Zhu \cite[Proposition 3.3]{JLZ16}. Moreover by the assumption $E(u)\leq\varepsilon_0$ and the reflection construction we have
\begin{equation}\label{omega02}
\Vert \nabla \tilde{u} \Vert_{L^2(D)}\,\leq\, C\sqrt{\varepsilon_0} \quad \text{and}\quad \Vert \omega \Vert_{L^2(D)}\,\leq\, C\sqrt{\varepsilon_0}\,.
\end{equation}
\vskip 1.5mm

Then using the {U}hlenbeck-{R}ivi\`ere decomposition method developed by Rivi\`ere \cite{Ri07} and Rivi\`ere-Struwe \cite{RS08} in the study of regularity of elliptic PDEs with antisymmetric structure, we can obtain a constant $\alpha>0$ such that (cf. \cite[page 50]{Ri12})
\begin{equation}\label{morreyest01}
\sup_{0<\rho<\frac{1}{4}, \,p\in D_{\frac{1}{2}}} \rho^{-\alpha}\int_{D_\rho(p)} \vert \Delta \tilde{u} \vert \,dx\,\leq \, C \sqrt{\varepsilon_0}\,.
\end{equation}

\begin{proof}[Proof of (\ref{morreyest01})]:
To see this, we first note that since $\tilde{u} - \overline{\tilde{u}}$ also satisfies \eqref{eq01} with the same $Q$ and $\omega$, where $\overline{\tilde{u}} = \frac{1}{\pi} \int_{D} \tilde{u} dx$ is the average of $\tilde{u}$ on $D$, without loss of generality we may assume that $\overline{\tilde{u}} =0$. Then by the work of Sharp \cite{Sh14} (see e.g. \cite[Corollary 1.4, Proposition 3.1]{Sh14}) on the higher integrability for solutions to a system of PDEs similar to \eqref{eq01}, we know that there exists a constant $C>0$ such that
\begin{equation}\label{higherint01}
\Vert \nabla^2 \tilde{u} \Vert_{L^1(D_{\frac{1}{2}})} \leq C\Vert \tilde{u}\Vert_{L^1(D)} \leq C \Vert \nabla \tilde{u}\Vert_{L^1(D)} \leq C\sqrt{\varepsilon_0}\,,
\end{equation}
see \cite[Theorem 1.2]{SZ16} and \cite[Theorem 2.4]{JLZ16}. Here we have used the Poincar\'e-Wirtinger inequality in the second inequality. We remark that this higher integrability \eqref{higherint01} essentially follows from the stability of the local Hardy space $h^1(D)$ (see e.g. Lamm and the first author \cite[Section A.2]{LL13}) under multiplication by H\"{o}lder continuous functions, coupled with the H\"{o}lder continuity of $\tilde{u}$ in $D$ proved by Scheven \cite[Theorem 4.1]{Scheven06}. 
\vskip 1.5mm

Then by the continuous embedding of $W^{1,1}(D)$ into $L^{2,1}(D)$, where $L^{2,1}$ is the Lorentz space (see e.g. \cite{Hu66}, \cite{Po83}, \cite{Tar98}, \cite{He02}), we have
\begin{equation}
\Vert \nabla \tilde{u} \Vert_{L^{2,1}(D_{\frac{1}{2}})} \leq C \Vert \nabla \tilde{u} \Vert_{W^{1,1}(D_{\frac{1}{2}})} \leq C \Vert \nabla \tilde{u}\Vert_{L^1(D)} \leq C\sqrt{\varepsilon_0}\,.
\end{equation}

Now for any $p\in D_{\frac{1}{2}}$ and $0<r<\frac{1}{2}$, we use the Hodge decomposition (see e.g. \cite[Corollary 10.5.1]{IM01}) to find $A\in W^{1,2}(D_r(p), \mathbf{R}^n)$ and $B\in W_0^{1,2}(D_r(p), \mathbf{R}^n)$ such that
\begin{equation}\label{eq07}
Q\nabla \tilde{u} = \nabla A + \nabla^\perp B \quad \text{in } D_r(p)\,,
\end{equation}
where $\nabla^\perp := (-\partial_{x_2}, \partial_{x_1})$ and we have
$$
\Vert \nabla A\Vert_{L^2(D_r(p))} + \Vert \nabla B\Vert_{L^2(D_r(p))} \leq \tilde{C} \Vert \nabla \tilde{u}\Vert_{L^2(D_r(p))} \,.
$$
Then we have (taking divergence on both sides of the first equation in \eqref{eq07})
\begin{equation}\label{eq08}
     \Delta A\, = \,\text{div} (Q   \nabla \tilde{u} ) = \omega \cdot Q\nabla \tilde{u}\quad  \text{in }\, D_r(p)\,,
     \end{equation}
and (taking curl on both sides of the first equation in \eqref{eq07})
\begin{equation}\label{eq09}
\left\{
   \begin{aligned}
    \Delta B\,& = \,\nabla^\perp Q \cdot  \nabla \tilde{u}&& \text{in }\, D_r(p)\,, \\
    B\,&=\, 0 && \text{on }\, \partial D_r(p)\,,\\
   \end{aligned}
 \right.
 \end{equation}

Now let $A = A_1 + A_2$ on $D_r(p)$ so that $\Delta A_2 =0$ and
\begin{equation}\label{eq10}
\left\{
   \begin{aligned}
    \Delta A_1\,& = \,\text{div} (Q   \nabla \tilde{u} )  \quad  &&\text{in }\, D_r(p)\,,\\
    A_1\,&=\, 0 && \text{on }\, \partial D_r(p)\,.\\
   \end{aligned}
 \right.
 \end{equation}
Then using \cite[Theorem 3.3.3]{He02} (which implies that the standard $L^p$-theory extends to Lorentz spaces), we furthermore get
\begin{equation}
\Vert \nabla A_1\Vert_{L^{2,1}(D_r(p))} \leq C \Vert Q \nabla \tilde{u}\Vert_{L^{2,1}(D_r(p))} \leq C \Vert \nabla \tilde{u}\Vert_{L^{2,1}(D_r(p))}\,.
\end{equation}
Hence 
by \cite[Theorem 3.3.4]{He02} we conclude that
\begin{equation}
\Vert A_1\Vert_{L^\infty (D_r(p))} \leq C \left(\Vert A_1 \Vert_{L^{2,1}(D_r(p))} + \Vert \nabla A_1 \Vert_{L^{2,1}(D_r(p))}\right) \leq C \Vert \nabla A_1\Vert_{L^{2,1}(D_r(p))} \leq C\sqrt{\varepsilon_0}\,,
\end{equation}
where we used again \cite[Theorem 3.3.3]{He02}  (which ensures that Poincar\'e's inequality extends to Lorentz spaces) and the fact that $A_1 = 0 $ on $\partial D_r(p)$ in the second estimate. Thus, an integration by parts yields
\begin{align}\label{a_1est}
\Vert \nabla A_1 \Vert^2_{L^2(D_r(p))} &= -\int_{D_r(p)} A_1 \Delta A_1 \,dx \leq \Vert A_1\Vert_{L^\infty (D_r(p))} \int_{D_r(p)} \vert \omega \cdot Q\nabla \tilde{u} \vert \,dx \notag\\
&\leq C\sqrt{\varepsilon_0} \Vert \nabla \tilde{u}\Vert^2_{L^2(D_r(p))}\,,
\end{align}
where we have used \eqref{omega02} and the fact that $Q\in W^{1,2}\cap L^\infty (D, GL(n,\mathbf{R}))$. Note that again by integration by parts (using $A_1 =0$ on $\partial D_r(p)$ and $\Delta A_2 =0$) we have
\begin{equation}\label{a_2est}
\Vert \nabla A_2 \Vert^2_{L^2(D_r(p))} = \Vert \nabla A \Vert^2_{L^2(D_r(p))}  - \Vert \nabla A_1 \Vert^2_{L^2(D_r(p))}  \leq  \tilde{C} \Vert \nabla \tilde{u} \Vert^2_{L^2(D_r(p))}\,.
\end{equation}
Now since $A_2$ is harmonic on $D_r(p)$ we know that for every $q \in D_r(p)$ the function
$$
\rho \to \frac{1}{\rho^2} \int_{D_\rho (q)}\vert \nabla A_2 \vert^2\, dx
$$
is increasing, see e.g. \cite[Lemma IV.1]{Ri12}. Now let $\bar{C}>0$ be such that 
$\Vert Q^{-1}\Vert_{L^\infty(D)} \leq \bar{C}$ and
$$
\delta = \min \left\{\frac{1}{4\sqrt{\tilde{C} \bar{C}}}, \frac{1}{2}\right\}\,,
$$
where $\tilde{C}$ is from \eqref{a_2est}. Then we have
\begin{equation}\label{a_2est02}
\int_{D_{\delta r} (p)} \vert \nabla A_2 \vert^2 \,dx\,\leq \, \frac{1}{16\tilde{C} \bar{C}} \int_{D_r(p)} \vert \nabla A_2 \vert^2 \,dx\leq \frac{1}{16\bar{C}} \Vert \nabla \tilde{u} \Vert^2_{L^2(D_r(p))}\,.
\end{equation}

Now using \eqref{eq09} and by the results of Coifman-Lions-Meyer-Semmes \cite{CLMS93} or Wente's lemma \cite{We69}, we know 
\begin{equation}\label{wente02}
\Vert B \Vert_{L^\infty(D_r(p))} + \Vert \nabla B\Vert_{L^{2,1} (D_r(p))} \leq C\Vert \nabla \tilde{u}\Vert^2_{L^2(D_r(p))}\,.
\end{equation}

Therefore, combining \eqref{eq07}, \eqref{a_1est}, \eqref{a_2est02} and \eqref{wente02} we have
\begin{align}
\Vert \nabla \tilde{u}\Vert^2_{L^2(D_{\delta r} (p))} &\leq  \Vert Q^{-1}\Vert_{L^\infty(D)} \left(2 \Vert \nabla A_1 \Vert^2_{L^2(D_r (p))}+ 2 \Vert \nabla A_2 \Vert^2_{L^2(D_{\delta r} (p))} +  \Vert \nabla B \Vert^2_{L^2(D_r (p))}\right)\\
&\leq C\sqrt{\varepsilon_0} \Vert \nabla \tilde{u}\Vert^2_{L^2(D_r (p))} + \frac{1}{8} \Vert \nabla \tilde{u}\Vert^2_{L^2(D_r (p))} + C\Vert \nabla \tilde{u}\Vert^4_{L^2(D_r (p))} \,.
\end{align}
Choosing $\varepsilon_0$ sufficiently small we arrive at
\begin{equation}
\Vert \nabla \tilde{u}\Vert^2_{L^2(D_{\delta r} (p))}  \leq \frac{1}{2} \Vert \nabla \tilde{u}\Vert^2_{L^2(D_r (p))} 
\end{equation}
for any $p\in D_{\frac{1}{2}}$ and $0<r<\frac{1}{2}$. Iterating this inequality gives the existence of a constant $\alpha>0$ such that for all $p\in D_{\frac{1}{2}}$ and all $\rho<\frac{1}{4}$, there holds
\begin{equation}
 \rho^{-2\alpha}\int_{D_\rho(p)} \vert \nabla \tilde{u} \vert^2 \,dx\,\leq C\int_{D} \vert \nabla \tilde{u} \vert^2 \,dx \leq \, C \varepsilon_0\,.
\end{equation}
Now by \eqref{eq01} we have
$$
\Delta \tilde{u} \,=\, Q^{-1}\left(-\nabla Q \nabla \tilde{u} + \omega\cdot Q \nabla \tilde{u}\right)\,.
$$
Then using \eqref{eq11}, for all $p\in D_{\frac{1}{2}}$ and all $\rho<\frac{1}{4}$ we have
\begin{align}
 \rho^{-\alpha}\int_{D_\rho(p)} \vert \Delta \tilde{u} \vert\,dx\,& \leq  C\rho^{-\alpha}\int_{D_\rho(p)} \vert \nabla Q\vert \vert \nabla \tilde{u} \vert + \vert \nabla \tilde{u} \vert^2 \,dx\notag\\
&\leq  C\rho^{-\alpha} \left[\left(\int_{D_\rho(p)} \vert \nabla Q \vert^2\,dx \right)^{\frac{1}{2}}  \left(\int_{D_\rho(p)} \vert \nabla \tilde{u} \vert^2\,dx \right)^{\frac{1}{2}} + \int_{D_\rho(p)} \vert \nabla \tilde{u} \vert^2\,dx \right]\\
&\leq C\sqrt{\varepsilon_0}\,.\notag
\end{align}
Here we have used $\Vert \nabla Q \Vert_{L^2(D)} \leq C$ for some $C>0$ depending only on $\mathcal{N}$ and $\Gamma$. This proves \eqref{morreyest01}.
\end{proof}

Now using \eqref{morreyest01}, a classical estimate on Riesz potentials then gives, for all $p\in D_{\frac{1}{4}}$ we have
\begin{equation}
 |\nabla \tilde{u}|(p)\le C \frac{1}{|x-p|}\ast \chi_{D_{\frac{1}{2}}}\ |\Delta \tilde{u}|+C \Vert \nabla \tilde{u}\Vert_{L^2(D)},
\end{equation}
where $\chi_{D_{\frac{1}{2}}}$ is the characteristic function of the ball $D_{\frac{1}{2}}$. Together with injections proved by Adams in \cite{Ad}, see also exercise 6.1.6 of \cite{Gra2}, the latter shows that
$$
 \Vert \nabla \tilde{u} \Vert_{L^{r}\left(D_{\frac{1}{4}}\right)} \leq C \sqrt{\varepsilon_0}
$$
for some $r>2$. Re-injecting this into the equation \eqref{eq01} and bootstrapping the estimates, we get 
$$\Vert \nabla \tilde{u} \Vert_{L^{\infty}\left(D_{\frac{1}{8}}\right)} \leq C \sqrt{\varepsilon_0}\,.$$
In particular, $\vert \nabla \tilde{u}\vert(0)\leq C \sqrt{\varepsilon_0}$. Then by a scaling argument we get the desired estimate \eqref{ptest}.
\end{proof}

Now we are ready to prove Theorem \ref{T:convexity2}.

\begin{proof}(of Theorem \ref{T:convexity2})
 In order to prove the energy convexity \eqref{Convexity1}, it suffices to show
\begin{equation}\label{TEMP1}
\Psi\,\geq\, -\frac{1}{2}\int_{D^+} |\nabla (v- u)|^2\,dx\,\,,
\end{equation}
where (using the boundary conditions and the harmonic map equation)
\begin{align}
 \Psi := &\int_{D^+} |\nabla v|^2 \,dx\,-\int_{D^+}|\nabla u|^2\,dx\,-\int_{D^+} |\nabla (v- u)|^2\,dx\, \notag\\
 = &\, 2\int_{D^+} \langle \nabla (v - u) ,\nabla u\rangle\,dx\,   \notag\\
 = &\,2\int_{D^+} \langle v - u ,  \,A(u)(\nabla u, \nabla u)\rangle\,dx\, + 2 \int_{\partial^C} \left\langle v - u , \,\frac{\partial u}{\partial \nu}\right\rangle\,ds\,\,.\label{Convexity2000}
\end{align}
Here $\nu = (0,-1)$ is the outward unit normal to $\partial^C$. Now we note that for any $p, q \in \mathcal{N}$ (or $\Gamma$ respectively), there exists a constant $C>0$, depending only on $\mathcal{N}$ (or $\Gamma$ respectively), such that $\left|(p-q)^{\perp}\right|\leq C |p-q|^2$, where the superscript $\perp$ denotes the normal component of a vector (see e.g. \cite[Lemma A.1]{CM08-1}). Therefore, using 
$$-\Delta u = A(u)(\nabla u, \nabla u) \perp T_u \mathcal{N}\,,$$ 
the Cauchy-Schwarz inequality together with \eqref{Convexity2000} yield
\begin{equation}\label{esti1}
\Psi \geq   - C_1\int_{D^+} |v - u|^2|\nabla u|^2\,dx\, - C_2 \int_{\partial^C} |v - u|^2 \left\vert\frac{\partial u}{\partial \nu}\right\vert\,ds\,,
\end{equation}
where $C_1>0$ depends only on $\mathcal{N}$ and $C_2>0$ depends only on $\Gamma$.  Now by Lemma \ref{hardy}, Theorem \ref{e-reg} and using the facts that $v=u$ on $\partial^A$ and $x_2=0$ on $\partial^C = [-1,1]\times \{0\}$, we have
\begin{align*}
\Psi &\geq   - C_3 \varepsilon_0\int_{D^+} \frac{|v - u|^2}{(1-|x|)^2}\,dx\, - C_4\sqrt{\varepsilon_0} \int_{\partial^C} \frac{|v - u|^2}{1-|x_1|}\,ds\\
&\geq  - 4C_3 \varepsilon_0\int_{D^+} |\nabla (v - u)|^2\,dx\, + C_4\sqrt{\varepsilon_0} \int_{-1}^1 \int_0^{\sqrt{1-x_1^2}} \partial_{x_2}\left(\frac{|v - u|^2}{1-|x_1|}\right)\,dx_2 dx_1\\
&=- 4C_3 \varepsilon_0\int_{D^+} |\nabla (v - u)|^2\,dx\, + C_4\sqrt{\varepsilon_0} \int_{D^+}\frac{ \langle v - u,  (v-u)_{x_2}\rangle}{1-|x_1|}\,dx\\
&\geq- 4C_3 \varepsilon_0\int_{D^+} |\nabla (v - u)|^2\,dx\, -\frac{1}{4}\int_{D^+} |\nabla (v - u)|^2\,dx - C_5\sqrt{\varepsilon_0} \int_{D^+} \frac{|v - u|^2}{(1-|x|)^2}\,dx\\
&\geq- C_6 \sqrt{\varepsilon_0}\int_{D^+} |\nabla (v - u)|^2\,dx\, -\frac{1}{4}\int_{D^+} |\nabla (v - u)|^2\,dx\,.
\end{align*}
In the second to the last inequality we have used Young's inequality. Choosing $\varepsilon_0$ sufficiently small so that $C_6 \sqrt{\varepsilon_0} \leq \frac{1}{4}$ we get the desired estimate \eqref{TEMP1}.
\end{proof}

\section{Existence and regularity of partial free boundary harmonic replacement}
\label{S:Existence and regularity of free boundary harmonic replacement}

In this section, we discuss the existence and regularity of the (partial) free boundary version of the harmonic replacement. The fixed boundary harmonic replacement was discussed in Colding-Minicozzi's work \cite{CM08}.

\subsection{Continuity of weakly harmonic maps with partial free boundary}
\label{SS:continuity of weakly free boundary harmonic maps}
In this section, we prove the $C^0$ regularity up to the boundary for weakly harmonic maps with mixed Dirichlet and free boundaries and small energy. For weakly harmonic maps with fixed continuous Dirichlet boundary, this $C^0$ boundary regularity was proved by Qing \cite{Qing93}. 
\vskip 1.5mm

Let $u: (D^+, \partial^C)\to (\mc N, \Gamma)$ be a weakly harmonic map with mixed Dirichlet and free  boundaries and small energy $E(u)\leq \varepsilon_0$, where $\varepsilon_0$ is from Theorem \ref{e-reg}, such that the Dirichlet boundary $u: \partial^A\to \mc N$ is \textit{continuous}. Then by Helein's interior regularity (see e.g. \cite{He02}), $u$ is smooth in $D^+$; by Qing's (Dirichlet) boundary regularity \cite{Qing93}, $u$ is continuous up to the interior of the Dirichlet boundary on $\partial ^A$; moreover, by the (free) boundary regularity result of Scheven \cite{Scheven06}, $u$ is also smooth up to the interior of the free boundary on $\partial^C$. So the only thing left to be verified is that $u$ is continuous up to the two corner points $(1,0)$ and $(1,\pi)$ (in polar coordinate of $\overline{D^+}$), i.e. the end points of the Dirichlet boundary on $\partial^A$ or the free boundary on $\partial^C$.\vskip 1.5mm

It suffices to prove the continuity around $p=(1,\pi)${\footnote{Given Scheven's reflection construction and equation \eqref{eq01}, one may also use {U}hlenbeck-{R}ivi\`ere decomposition method as in M\"uller-Schikorra \cite{MuSc09} to prove boundary regularity.}. We will first prove that there exists a sequence of points $\{x_i\}$ on $\partial^C$ converging to $p$, such that the $u(x_i)$'s are all close to $u(p)$; next we show that all the intermediate points in $u([x_i, x_{i+1}])$ are also close to $u(p)$. For the convenience of later proof, we will parallelly shift $D^+$ in $\mathbf{R}^2$ such that $p=(0,0)$. Now $D^+=\{(x,y)\in \mathbf{R}^2:y > 0,(x-1)^2+y^2 <1\}$ and $\partial^A = \{(x,y)\in \mathbf{R}^2:y > 0,(x-1)^2+y^2 =1\}$, and we will use polar coordinates on $\mathbf{R}^2$.\vskip 1.5mm

We first need a variant of the Courant-Lebesgue Lemma (c.f. \cite[Lemma 2]{Qing93}) at the corner point $p$. 

\begin{lemma}\label{L:Courant-Lebesgue at Corner}
Given $0<l<\frac{1}{2}$, there exists $l'\in(l,2l)$ such that 
\begin{equation}
\int_0^{\alpha(l')}\left\vert\frac{\partial u}{\partial\theta}\right\vert^2(l',\theta)d\theta\leq\frac{1}{\log 2}E\left(u\vert_{(D_{2l}(0)\setminus D_l(0))\cap D^+}\right),
\end{equation}
where $\alpha(l')$ is the angle such that $(l',\alpha(l'))\in \partial^A$. Consequently, we have
\begin{equation}
\left\vert u(l',\theta_1)-u(l',\theta_2)\right\vert\leq\sqrt{\frac{N}{\log 2}}E\left(u\vert_{(D_{2l}(0)\setminus D_l(0))\cap D^+}\right)^{1/2}\vert\theta_1-\theta_2\vert^{1/2},
\end{equation}
where $N$ is the dimension of the ambient space of the embedding of $\mathcal{N}$ into $\mathbf{R}^N$.
\end{lemma}

\begin{proof}
The proof is virtually the same as the proof of \cite[Lemma 2]{Qing93}. 
\[E\left(u\vert_{(D_{2l}(0)\setminus D_l(0))\cap D^+}\right)\geq\int_{l}^{2l}\int_{0}^{\alpha(r)}\left(\left\vert\frac{\partial u}{\partial r}\right\vert^2+\frac{1}{r^2}\left\vert\frac{\partial u}{\partial\theta}\right\vert^2\right)rdrd\theta,\]
where $\alpha(r)$ is the angle such that $(r,\alpha(r)) \in \partial^A = \{(x,y)\in \mathbf{R}^2:y > 0,(x-1)^2+y^2 =1\}$. Therefore there exists $l'\in(l,2l)$ such that
\[\int_0^{\alpha(l')}\left\vert\frac{\partial u}{\partial\theta}\right\vert^2(l',\theta)d\theta\leq\frac{1}{\log2}E\left(u\vert_{(D_{2l}(0)\setminus D_l(0))\cap D^+}\right).\]
Then we have
\[\vert u(l',\theta_1)-u(l',\theta_2)\vert\leq \left \vert\int_{\theta_1}^{\theta_2}\frac{\partial u}{\partial \theta}(l',\theta)d\theta \right\vert\leq\sqrt{\frac{N}{\log 2}}E\left(u\vert_{(D_{2l}(0)\setminus D_l(0))\cap D^+}\right)^{1/2}\vert\theta_1-\theta_2\vert^{1/2}.\qedhere\]
\end{proof}

Now we are ready to prove the continuity up to the corner on the boundary of $D^+$.

\begin{thm}
\label{T: Continuity of free boundary harmonic map}
There exists $\varepsilon_0>0$ such that if $u: (D^+, \partial^C)\to (\mc N, \Gamma)$ is a weakly harmonic map with mixed Dirichlet and free boundaries and $E(u)\leq\eps_0$, and $u$ is continuous on $\partial^A$, then $u$ is continuous on $\overline{D^+}$.
\end{thm}

\begin{proof}
Since $u$ is continuous on $\partial^A$, by the discussion before Lemma \ref{L:Courant-Lebesgue at Corner}, it suffices to show that given $\eps>0$, there exists $\delta>0$ such that for $x\in \partial^C$ and $\vert x-p\vert\leq\delta$ we have $\vert u(x)-u(p)\vert\leq\eps$. By the continuity of $u$ on $\partial^A$, we can find $\delta_1>0$ such that if $p_l=(l,\alpha(l))\in \partial^A$ and $\vert p_l-p\vert\leq\delta_1$, then $\vert u(p_l)-u(p)\vert\leq\eps/10$. 
\vskip 1.5mm

Now we identify the unit interval $[0, 1]$ with half of $\partial^C=[-1, 1]\times\{0\}\subset \mathbf{R}^2$ and abuse the notation such that $x\in [0, 1]$ represents a point on $\partial^C$. Consider the decomposition
\[[0,1]=\cup_{k=0}^\infty\left[2^{-k-1}, 2^{-k}\right]=\cup_{k=0}^\infty I_k.\] 
Note that we have
\[\sum_{k=0}^\infty E\left(u\vert_{(D_{2^{-k-1}}\setminus D_{2^{-k}}) \cap D^+ }\right)\leq\eps_0.\]
Then there is some $K_1>0$ such that for all $k\geq K_1$ we have $E\left(u\vert_{(D_{2^{-k-1}}\setminus D_{2^{-k}})\cap D^+}\right)\leq \eps^2/(1000 N^2)$. Now using Lemma \ref{L:Courant-Lebesgue at Corner}, for any $k\geq K_1$ we can pick $x_k\in I_k$ such that 
$$\vert u(x_k)-u(p)\vert\leq \left\vert u(x_k)-u(p_{|x_k|})\right\vert + \left\vert u(p_{|x_k|})-u(p)\right\vert \leq \eps/8\,.$$

Now let us consider a family of half disks ${D^+_k} \subset D^+$, where $D^+_k$ has corner points $(2^{-k-2},0)$ and $(2^{-k}+2^{-k-2},0)$. Then the center of $D^+_k$ is $y_k := (2^{-k-1}+2^{-k-2}, 0)$ and the radius of $D^+_k$ is $r_k:= 2^{-k-1}$. Since each $D^+_k$ can only overlap with at most four other $D^+_k$'s, we have
\[\sum_{k=0}^\infty    E\left(u\vert_{D^+_k }\right)  \leq 4\eps_0.\]
Thus, for any $\eps_0>0$, there exists $K_2>0$ such that for any $k\geq K_2$ we have $E\left(u\vert_{D^+_k } \right)\leq \eps_0$. Then by rescaling and Theorem \ref{e-reg}, for any $y\in[2^{-k-1},2^{-k}]$ with $k\geq K_2$ we have
\begin{equation}
\vert\nabla u (y)\vert\leq C2^{k+1}\sqrt{\eps_0}.
\end{equation}
Then integration along the interval between $y$ and $x_k$ gives
\[\vert u(y)-u(x_k)\vert\leq 2C\sqrt{\eps_0}.\]
If we choose $\eps_0$ small enough (depending only on $\eps$), then for any $y\in[2^{-k-1},2^{-k}]$ with $k\geq K_2$ we have $\vert u(y)-u(x_k)\vert\leq \eps/4$. So for all $y\in(0,2^{-K})$ where $K = \max \{K_1, K_2\}$, we have
\[\vert u(y)-u(p)\vert\leq \eps.\]
Then we conclude the continuity of $u$ at $p$ and hence finish the proof of the theorem.
\end{proof}

As a corollary of Theorem \ref{T: Continuity of free boundary harmonic map}, we have
\begin{cor}
A weakly harmonic map $u$ with mixed Dirichlet and free boundaries on $D^+$ is continuous on $\overline{ D^+}$ provided that $u$ is continuous on $\partial^A$.
\end{cor}

\subsection{Existence of partial free boundary harmonic replacement}

\begin{thm}
\label{T:existence and regularity of free boundary harmonic replacement}
There exists $\varepsilon_0>0$ such that for any $v \in C^{0}(\overline{D^+}, \cN) \cap W^{1,2} (\overline{D^+}, \cN)$ with $v(\partial^C)\subset \Gamma$ and $E(v)\leq \varepsilon_0$, there exists a unique harmonic map $u\in C^0(\overline{D^+}, \cN) \cap W^{1,2}(D^+, \cN)$ such that $E(u)\leq \varepsilon_0$, $u=v$ on $\partial^A$, $u(\partial^C)\subset \Gamma$ and $u$ meets $\Gamma$ orthogonally along $\partial^C$.
\end{thm}

\begin{rem}
\label{Re:free boundary harmonic replacement}
The map $u$ in Theorem \ref{T:existence and regularity of free boundary harmonic replacement} is usually called the (partial) {\em free boundary harmonic replacement} of $v$.
\end{rem}

\begin{proof} (of Theorem \ref{T:existence and regularity of free boundary harmonic replacement})
Combining H\'elein's interior regularity for weakly harmonic maps on two dimensional domains, Qing's (Dirichlet) boundary regularity \cite{Qing93},  Scheven's free boundary regularity (\cite{Scheven06}) and Theorem \ref{T: Continuity of free boundary harmonic map}, we know that the (partial) free boundary harmonic replacement $u$ of $v$, if it exists, is smooth in $D^+ $ and is continuous in $\overline{D^+}$. \vskip 1.5mm

The (partial) free boundary harmonic replacement can be constructed as follows:
suppose $v\in C^0(\overline{D^+}, \cN) \cap W^{1,2}(D^+, \cN)$ is such that $v(\partial^C)\subset \Gamma$ and $E(v)\leq\varepsilon_0$.  Let us define the space $\mc F$ to be the space of maps:
 \[\mc F:=\{w\in
 W^{1,2}(D^+, \mathcal{N}):\, w\vert_{\partial^A}=v\vert_{\partial^A},w(\partial^C)\subset \Gamma\}. \] 
Now we choose an energy minimizing sequence $u^i\in\mc F$, i.e.
\[ \lim_{i\to\infty}E(u^i)=\inf\{E(w):\, w\in\mc F\}. \] 
By Rellich compactness theorem, we can find a subsequence of $\{u^i\}$ that weakly converges to a $W^{1,2}$ map $u$. By first variation of energy functional at $u$ we know that $u$ is a weakly harmonic map with mixed Dirichlet and free boundaries on $D^+$. Then indeed we know that $\{u^i\}$ converges strongly in $W^{1,2}$ to $u$ by the energy convexity Theorem \ref{T:convexity2} and Poincar\'e inequality Lemma \ref{L:Poincare inequality with partial zero boundary values}. \end{proof}

\section{Overview of the variational approach}
\label{S:Overview of the variational approach}
In this section, we provide an overview for the proof of Theorem \ref{T:main1}. \vskip 1.5mm

As in the proof of the classical Plateau problem, the area functional is too weak to control the maps, so we have to change gear to the energy functional. The next result says that one can take an approximating sequence of sweepouts in a given homotopy class, so that their maximal energy converges to the maximal area. In particular, we have:

\begin{thm}
\label{T:area width equals energy width}
Given $\beta\in \Omega$ with $W = W(\Om_\beta)>0$, there exists a sequence of sweepouts $\gamma^j\in\Omega_\beta$ with 
\[ \max_{s\in[0,1]}E(\gamma^j(\cdot,s))\searrow W. \]
\end{thm}
\begin{remark*}
The proof is given in Section \ref{S:Conformal parametrization}.  Note that we have $W\leq \max_{s\in[0,1]}\area(\gamma^j(\cdot,s))\leq \max_{s\in[0,1]}E(\gamma^j(\cdot,s))\searrow W$.
\end{remark*}

The next ingredient is a tightening theorem. We first fix some notations concerning balls in $D$.
\begin{defi}
\label{D:generalized balls}
A {\em generalized ball} $B$ in the unique disk $D$ is either an interior ball of $D$, or the intersection with $D$ of a ball of $\mb R^2$ centered at some boundary point of $\partial D$. That is: $B\subset D$ or $B=B_r(x)\cap D$ where $x\in\partial D$ and $r<\frac{1}{2}$. We will call the first case a {\em classical ball} and the second case a {\em boundary ball} respectively. \vskip 1.5mm

Given $\rho>0$, when $B$ is a classical ball, we let $\rho B\subset D$ denote the ball with the same center of $B$ and radius $\rho$ times that of $B$; when $B$ is a boundary ball, we can define $\rho B\subset D$ as follows: there exists a unique fractional linear transformation $\Pi_B$ from $D$ to the upper half plane $\mb H^2$ so that $\Pi_B(B)=D^+$, and $\rho B$ is defined as $\Pi_B^{-1}(D^+_\rho)$.
\end{defi}
\begin{remark*}
Note that in the following, we will frequently identify a boundary ball $B$ with its image $\Pi_B(B)$ in $\mb H^2$.
\end{remark*}

The following result plays the role of deformation lemma in nonlinear analysis; (see \cite{Struwe08}).
\begin{thm}
\label{T:Tightening theorem}
Given $\beta\in \Omega$ with $W=W(\Om_\beta)>0$,  there exist a constant $\epsilon_1>0$ and a continuous function $\Psi:[0,\infty)\to[0,\infty)$ with $\Psi(0)=0$, both depending on $(\mc N, \Gamma)$,
so that given any $\tilde\gamma\in\Omega$ without non-constant harmonic slices other than $\gamma(\cdot,0),\gamma(\cdot,1)$, there exists $\gamma\in\Omega_{\tilde\gamma}$ such that $E(\gamma(\cdot,t))\leq E(\tilde\gamma(\cdot,t))$ for each $t$, and for each $t$ with $E(\tilde\gamma(\cdot,t))\geq W/2$, we have the following property:
\begin{itemize}
\item[($\ast$)] If $\mc{B}$ is any finite collection of disjoint generalized closed balls in $D$ with 
\[\int_{\cup_{\mc{B}}B}\vert\nabla\gamma(\cdot, t)\vert^2<\eps_1\]
and if $v:\cup_{\mc{B}}\frac{1}{8}B\to \mc N$ is the free boundary harmonic replacements of $\gamma(\cdot, t)$ on $\cup_{\mc{B}}\frac{1}{8} B$, 
then
\[ \int_{\cup_{\mc{B}}\frac{1}{8}B}\vert\nabla \gamma(\cdot,t)-\nabla v\vert^2\leq\Phi(E(\tilde\gamma(\cdot,t))-E(\gamma(\cdot,t))). \]
\end{itemize}
\end{thm}
\begin{remark*}
The proof is given in Section \ref{S:Construction of the tightening process}.
\end{remark*}

We also need the following compactness result. Let $\eps_{SU}$ and $\eps_{F}$ be the small thresholds (depending only on $(\mc N, \Gamma)$) from \cite[Main Estimate 3.2]{Sacks-Uhlenbeck81} and \cite[Proposition 1.7]{Fraser00} respectively, so that we can get uniform interior estimates for harmonic maps or free boundary harmonic maps with energy less than $\eps_{SU}$ and $\eps_{F}$ respectively.

\begin{thm}
\label{T:Compactness theorem}
Suppose $\epsilon_1, E_0>0$ with $\epsilon_1<\min\{\eps_{SU}, \eps_F\}$ and $u^j: (D, \partial D)\to (\cN, \Gamma)$ is a sequence of maps in $C^0(\overline{D}, \cN) \cap W^{1,2}(D, \cN)$ with $E_0\geq E(u^j)$ satisfying
\begin{equation}
\label{E:area assumption in compactness theorem}
 \area(u^j)>E(u^j)-\frac{1}{j}, 
\end{equation}
and:
\begin{itemize}
\item[($\dagger$)] for any finite collection $\mc{B}$ of disjoint generalized closed balls in $D$ with 
\[\int_{\cup_{\mc{B}}B}\vert\nabla u^j\vert^2<\epsilon_1,\]
if $v: \cup_{\mc{B}}\frac{1}{8}B\to \mc N$ is the free boundary harmonic replacements of $u_j$ on $\cup_{\mc{B}}\frac{1}{8}B$, then
\[\int_{\cup_{\mc{B}}\frac{1}{8}B}\vert\nabla u^j-\nabla v\vert^2\leq\frac{1}{j}.\]
\end{itemize}
Then a subsequence of the $u^j$'s varifold converges to a collection of harmonic maps $v_0, \cdots, v_m: (D, \partial D)\to (\cN, \Gamma)$ with free boundary and harmonic spheres $\tilde v_1, \cdots, \tilde v_l: \mathbf{S}^2\to \cN$. Moreover, the energy identity holds $\sum_{i=0}^m E(v_i)+\sum_{k=1}^l E(\tilde v_k)=\lim_{j\to\infty}E(u^j)$.
\end{thm}
\begin{remark*}
The proof is given in Section \ref{S:Compactness of maximal slices}. Note that $E(v_i)=\area(v_i)$ and $E(\tilde v_j)=\area(\tilde v_j)$.
\end{remark*}

Now we prove the main Theorem \ref{T:main1} with these three theorems.

\begin{proof}[Proof of Theorem \ref{T:main1}]
Choose a sequence $\tilde\gamma^j\in\Omega_\beta$ as in Theorem \ref{T:area width equals energy width}, and we assume that
\[\max_{t\in[0,1]}E(\tilde\gamma^j(\cdot,t))<W+\frac{1}{j}.\]
We can slightly change the parametrization so that $\tilde\gamma^j$ maps a small open subset of $D$ to a point so that each slice cannot be harmonic unless it is a constant map; (we refer details to \cite[footnote 8]{CM08}). Applying Theorem \ref{T:Tightening theorem} to $\tilde\gamma^j$ gives a sequence $\gamma^j\in\Omega_\beta$. We will show $\gamma^j$ has the desired properties. \vskip 1.5mm

Let us argue by contradiction. Let $\mc G^{W}$ be the set of collections of harmonic maps from $\mathbf{S}^2\to \mc N$ and free boundary harmonic maps $(D, \partial D)\to(\mc N, \Gamma)$ so that the sum of the energies is exactly $W$. Suppose $\{\gamma^j\}$ does not have the desired property, which means that there exists some $\eps>0$ such that there exist $j_k\to\infty$ and $s_k\in[0,1]$ with $d_V(\gamma^{j_k}(\cdot,s_k),\mc G^W)\geq\eps$ and $\area(\gamma^{j_k}(\cdot,s_k))>W-1/k$. Then by $E(u)\geq\area(u)$ we get $\lim_{k\to\infty}E(\gamma^{j_k}(\cdot,s_k))=W$, and 
\[ E(\tilde\gamma^{j_k}(\cdot,s_k))-E(\gamma^{j_k}(\cdot,s_k))\leq E(\tilde\gamma^{j_k}(\cdot,s_k))-\area(\gamma^{j_k}(\cdot,s_k))\leq\frac{1}{k}+\frac{1}{j_k}\to 0. \]
Since the tightening process decreases the energy, we get
\[E(\gamma^{j_k}(\cdot,s_k))-\area(\gamma^{j_k}(\cdot, s_k))\to 0, \mbox{ as $k\to\infty$}.\]

By Theorem \ref{T:Tightening theorem} we have that: if $\mc{B}$ is any finite collection of disjoint generalized closed balls in $D$ with 
\[\int_{\cup_{\mc{B}}B}\vert\nabla\gamma^{j_k}(\cdot, s_k)\vert^2<\eps_1,\]
and if $v:\cup_{\mc{B}}\frac{1}{8}B\to \mc N$ is the free boundary harmonic replacement of $\gamma^{j_k}(\cdot, s_k)$ on $\cup_{\mc{B}}\frac{1}{8}\partial B$, then
\[\int_{\cup_{\mc{B}}\frac{1}{8}B}\vert\nabla \gamma^{j_k}(\cdot, s_k)-\nabla v\vert^2\leq\Phi(\frac{1}{k}+\frac{1}{j_k})\to 0.\]

Therefore applying Theorem \ref{T:Compactness theorem} gives a subsequence of $\gamma^{j_k}(\cdot,s_k)$'s that varifold converges to a collection of desired harmonic disks with free boundary and harmonic spheres. The energy identity implies that the sum of the energies of the limit is exactly $W.$ However this contradicts $d_V(\gamma^{j_k}(\cdot,s_k),\mc G^W)\geq\eps$. This finishes the proof of Theorem \ref{T:main1}. 
\end{proof}

\section{Conformal parametrization}
\label{S:Conformal parametrization}

In this section, we prove Theorem \ref{T:area width equals energy width}. The main idea of the proof follows \cite[Appendix D]{CM08}; (see also \cite[Section 3]{Zhou10} and \cite[Section 3]{Zhou17}). For a given sweep-out, we will find conformal reparametrizations of the regularization of this sweep-out, so that the energy of each slice of this family can not be much larger than that of the area. 

\begin{proof}[Proof of Theorem \ref{T:area width equals energy width}]
First we claim that for a given sweep-out $\tilde\gamma(\cdot,t)\in\Omega_{\tilde\gamma}$, we can find a regularized sweep-out $\gamma(\cdot,t)\in\Omega_{\tilde\gamma}$ which lies in $C^0\big([0, 1], C^2(D,\mc N)\big)$ as a map of $t$, such that $\gamma(\cdot, t)$ is closed to $\tilde\gamma(\cdot, t)$ uniformly in the $W^{1,2}\cap C^0$-norm for all $t\in[0, 1]$. This follows from a standard argument using mollification just like \cite[Lemma D.1]{CM08}. Here we only point out necessary modifications of \cite[Lemma D.1]{CM08} to make sure that the images of $\partial D$ under each slice $\gamma(\cdot,t)$ lie in the constraint submanifold $\Gamma$. In particular, near the boundary $\partial D$ we can first enlarge the domain $D$ to $D_{1+\al}$ (for small $\al>0$) by reflecting $\tilde\gamma(\cdot, t)$ across $\Gamma$ in the Fermi coordinates around $\Gamma$ and then do classical mollification under the Fermi coordinates. The mollified maps when restricted to $D$ will map $\partial D$ to $\Gamma$ by our construction. By the classical mollification result \cite[page 250, Theorem 1]{Evans} and Lemma \ref{L:Fermi equivalent to Rn}, we can make sure that the obtained maps are closed to $\tilde\gamma(\cdot, t)$ uniformly in $W^{1,2}\cap C^0$ near $\partial D$. In the interior of $D$, we can just mollify $\tilde\gamma(\cdot, t)$ in $\mb R^N$. To combine them, we can chose a partition of unity to glue these two mollifications together, and by the same argument as \cite[page 252, Theorem 3]{Evans}, the obtained maps are also close to $\tilde\gamma(\cdot, t)$ uniformly in $W^{1,2}\cap C^0(D)$. Up to this step, under the obtained maps the boundary $\partial D$ goes into $\Gamma$, but the interior of $D$ may get out of $\mc N$. Finally, one can follow \cite[Lemma D.1]{CM08} to project these maps to $\mc N$ using the nearest point projection to get the desired $\gamma(\cdot, t)$. By choosing the mollification parameter small enough, we can make sure $\max_{t\in[0, 1]}\Vert \gamma(\cdot, t)-\tilde\gamma(\cdot, t)\Vert$ is as small as we want. Note that an explicit homotopy between $\gamma$ and $\tilde\gamma$ is given by letting the mollification parameter goes to $0$. So we finish the sketch of the proof of the claim.\vskip 1.5mm

Then $\gamma(\cdot,t)$ induce a continuous one-parameter family of $C^1$ metrics $g(t)=\gamma(t)^*(\text{metric on }\mc N)$ on $D$. These family of metrics may be degenerate, so we define the perturbed metrics as $\tilde g(t)=g(t)+\eps g_0$ where $g_0$ is the standard flat metric on $D$. Then by \cite[Lemma D.6]{CM08}, \cite[Lemma 3.6]{Zhou10}  and \cite[Proposition 3.1]{Zhou17}, we can construct a family of conformal reparametrizations $h_t:D_{g_0}\to D_{\tilde g(t)}$ (which fix three given points on $\partial D$) that varies continuously in $C^0\cap W^{1,2}(D, D)$. Then with the conformality we can control the energy:
\begin{equation}
\begin{split}
E(\gamma(\cdot,t)\circ h_t)&=E(h_t:D_{g_0}\to D_{g(t)})\\
&\leq E(h_t:D_{g_0}\to D_{\tilde g(t)})=\area(D_{\tilde g(t)})\\
&=\int_D(\det(g_0^{-1}g(t))+\eps Tr(g_0^{-1}g(t))+\eps^2)^{1/2} d\vol_{g_0}\\
&\leq \area(D_{g(t)})+\pi(\eps^2+2\eps\sup_t\vert g_0^{-1}g(t)\vert)^{1/2}.
\end{split}
\end{equation}
Choose $\eps>0$ so that $\pi(\eps^2+2\eps\sup_t\vert g_0^{-1}g(t)\vert)^{1/2}<\delta/2$, we get $E(\gamma(\cdot,t)\circ h_t)-\area(D_{g_t})\leq\delta/2$. Then if $\tilde\gamma^j$ is a sequence of sweepouts in $\Omega_\beta$, then $\gamma^j\circ h_t$ constructed as above is a desired sequence.
\end{proof}

\section{Construction of the tightening process}
\label{S:Construction of the tightening process}

This section is devoted to the proof of Theorem \ref{T:Tightening theorem}. 

\subsection{Continuity of harmonic replacement}
In this part we want to prove that the harmonic replacement process is actually continuous as a map from $C^0(\overline {D^+})\cap W^{1,2}(D^+)$ to itself if we restrict to small energy maps. This generalizes the continuity of harmonic replacement process on interior balls by Colding-Minicozzi \cite{CM08}.

\begin{thm}
\label{T:continuity of harmonic replacement}
Let $\eps_0$ be as in the Theorem \ref{T:convexity2} and set 
\[\mc M=\{u\in C^0(\overline {D^+}, \mc N)\cap W^{1,2}(D^+, \mc N):\, E(u)\leq \eps_0, u(\partial^C)\subset\Gamma\}.\]
Given $u\in\mc M$, let $w\in\mc M$ be the unique free boundary harmonic replacment of $u$ (produced by Theorem \ref{T:existence and regularity of free boundary harmonic replacement}), 
then $u\to w$ is continuous as a map from $C^0(\overline{D^+})\cap W^{1,2}(D^+)$ to itself.
\end{thm}

\subsubsection{$W^{1,2}$-continuity}
First we prove this map is $W^{1,2}$ continuous. In \cite{CM08}, Colding-Minicozzi's idea is to construct a comparison map by interpolating in $\mb R^N$ between two maps of the same boundary values in $W^{1, 2}(D, \mc N)$, and then project the interpolation back to $\mc N$. In free boundary case, however, such interpolation-projection trick may not leave the image of $\partial^C$ lying inside $\Gamma$. Here we prove the $W^{1, 2}$-continuity by contradiction. The main idea is that if the $W^{1,2}$-continuity fails, then we can find a sequence $u_i$ converging to $u_\infty$, but the replacements $w_i$'s of the $u_i$'s have energy strictly greater than the energy of $w_\infty$. In this scenario, we are able to construct some comparison maps $v_i$ sharing the same fixed boundary with $w_i$, but having energy strictly smaller than that of $w_i$ for $i$ large. This contradicts the uniqueness of small energy free boundary harmonic map; (see Corollary \ref{C:uniqueness of small energy free boundary harmonic maps}).\vskip 1.5mm

The key ingredient in our proof is the construction of certain new comparison maps. We first collect a few preliminary results for the construction. \vskip 1.5mm

There are five components in the comparison map. One of the components consists of ``small" free boundary harmonic maps. The following lemma shows that the energy of these maps are actually small.

\begin{lemma}
\label{L:small free boundary map}
There exists $\delta_0>0$ such that if a Lipschitz map $f:[0,\pi]\to \mc N$ satisfies $f(0)\in\Gamma,\, f(\pi)\in\Gamma$ and $\int_0^\pi\vert f'\vert\leq\delta\leq\delta_0$, then there exists a map $v:D^+\to \mc N$ such that $v(1,\theta)=f(\theta)$, $v(\partial^C)\subset\Gamma$ and $E(v)\leq C\delta^2$ for some universal constant $C$ depending only on $\mc N,\Gamma,\delta_0$.
\end{lemma}

\begin{proof}
By the gradient bound of $f$ and the fact $f(0), f(\pi)\in\Gamma$, we can assume that the image of $f$ lies in the Fermi coordinates $\{y^1, \cdots, y^n\}$ around $\Gamma$ where Lemma \ref{L:Fermi equivalent to Rn} in the appendix applies. We can further assume that locally $\Gamma$ is identified as a subset of $\{y^{k+1}=\cdots=y^n=0\}$ and $f(0)=0$. 
Define 
\[v(r,\theta)=r\cdot f(\theta)\]
where $r\cdot f(\theta)$ is the scalar multiplication under Fermi coordinates; hence $v$ is a map from $D^+\to \mc N$, and $v(r,0)$ and $ v(r,\pi)$ both lie in $\Gamma$ for any $r\in[0,1]$, thus $ v(\partial^C)\subset\Gamma$. \vskip 1.5mm

It is only left to check that the energy of $v$ is small. By Lemma \ref{L:Fermi equivalent to Rn}, we only need to check that the energy of $v$ is small as a map to the Euclidean space $\mb R^N$. In particular, Lemma \ref{L:Fermi equivalent to Rn} implies $\int_0^\pi\vert f'\vert\leq (1+\alpha)\delta$ under the Euclidean metric for some small $\alpha>0$. Also as $f(0)=0$, we get $\vert f(\theta)\vert\leq \pi(1+\alpha)\delta$ for $\theta\in[0,\pi]$. Now we have the energy estimates for $v$ as a map into standard Euclidean space $\mb R^N$: 
\begin{equation}
\begin{split}
E(v)&=\int_{0}^1\int_{0}^{\pi}(\vert \frac{\partial v}{\partial r}\vert^2+\frac{1}{r^2}\vert\frac{\partial v}{\partial \theta}\vert^2)rdrd\theta\\
&\leq \int_{0}^1\int_{0}^{\pi}(\vert f\vert^2+\vert f'\vert^2)rdrd\theta \,\leq\, C\delta^2,
\end{split}
\end{equation} 
for some universal constant $C>0$. Thus we finish the proof.
\end{proof}

Another component in our comparison map is a modified interpolation band. First we recall a lemma in \cite{CM08} to construct the interpolation of band between two circles.

\begin{lemma}
\label{L:Interpolation band}
There exists $\delta_0>0$ such that for $\delta<\delta_0$ the following statement holds. Let $f:[0,\pi]\to \mc N$, $g:[0,\pi]\to \mc N$ be two $C^0$ maps such that $\vert f-g\vert\leq\delta$, $\int\vert f'\vert^2\leq\delta'$, $\int\vert g'\vert^2\leq\delta'$. Then there exists $\rho\in (0,1/2]$ and a map $\tilde v:D^+\setminus D^+_{1-\rho}\to \mc N$ such that $\tilde v(1-\rho,\theta)=f(\theta)$, $\tilde v(1,\theta)=g(\theta)$, and $E(\tilde v)\leq C\delta^{1/2}\delta'^{1/2}$ for some constant $C>0$ depending only on $\mc N$.
\end{lemma}

\begin{proof}
We refer the proof to \cite[Lemma 3.11]{CM08}, where the construction works for the whole circles, but it can be generalized to half circles without any modification. Note our assumption here is even stronger.  
\end{proof}

With this construction on hand, we can construct an interpolation between two arcs on a modified band. Let us first define a modified band. A {\em modified band} $\textit{MB}_{a,b}$ is defined as 
\begin{equation}
\begin{split}
&\textit{MB}_{a,b}=\{(r,\theta): r\in[a,b],\theta\in[0,\pi]\}\setminus\\
&\left(\left\{(r,\theta):\vert (r,\theta)-(\frac{b+a}{2},0)\vert<\frac{b+a}{2}\right\}\cup \left\{(r,\theta):\vert (r,\theta)-(\frac{b+a}{2},\pi)\vert<\frac{b+a}{2}\right\} \right),
\end{split}
\end{equation}
see Figure \ref{pic1}. In the following context we may use $\phi$ to denote the angle parameter of the arcs of the removed disks.
\begin{figure}[ht]
\centering
\ifpdf
  \setlength{\unitlength}{1bp}%
  \begin{picture}(403.18, 114.40)(0,0)
  \put(0,0){\includegraphics{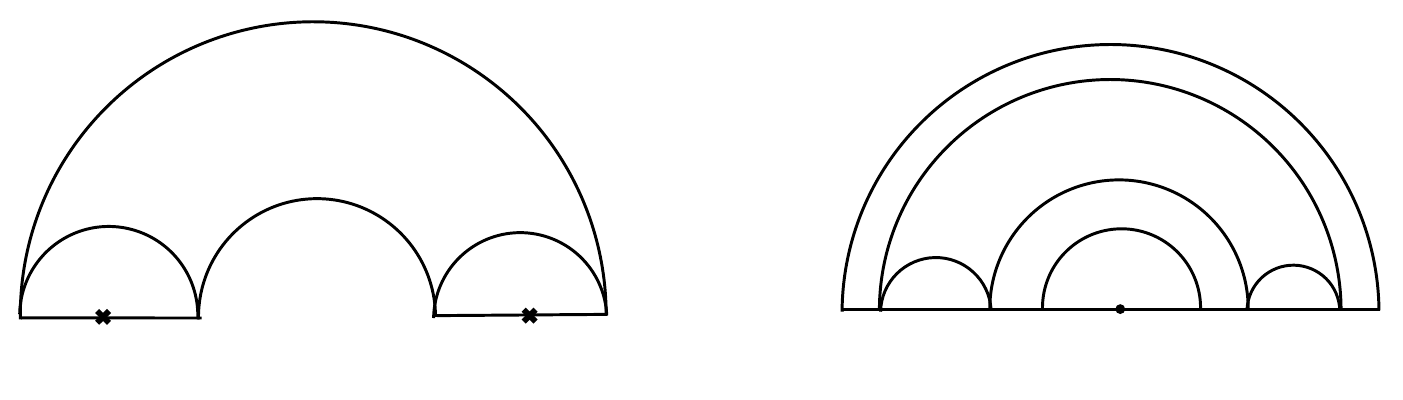}}
  \put(66.04,7.34){\fontsize{7.76}{9.31}\selectfont Modified Band}
  \put(263.02,28.15){\fontsize{6.47}{7.76}\selectfont $F_4$}
  \put(316.50,30.31){\fontsize{6.47}{7.76}\selectfont $F_1$}
  \put(341.79,40.43){\fontsize{6.47}{7.76}\selectfont $F_2$}
  \put(344.68,65.00){\fontsize{6.47}{7.76}\selectfont $F_3$}
  \put(367.08,28.51){\fontsize{6.47}{7.76}\selectfont $F_4$}
  \put(375.75,60.66){\fontsize{6.47}{7.76}\selectfont $F_5$}
  \end{picture}%
\else
  \setlength{\unitlength}{1bp}%
  \begin{picture}(403.18, 114.40)(0,0)
  \put(0,0){\includegraphics{1}}
  \put(66.04,7.34){\fontsize{7.76}{9.31}\selectfont Modified Band}
  \put(263.02,28.15){\fontsize{6.47}{7.76}\selectfont $F_4$}
  \put(316.50,30.31){\fontsize{6.47}{7.76}\selectfont $F_1$}
  \put(341.79,40.43){\fontsize{6.47}{7.76}\selectfont $F_2$}
  \put(344.68,65.00){\fontsize{6.47}{7.76}\selectfont $F_3$}
  \put(367.08,28.51){\fontsize{6.47}{7.76}\selectfont $F_4$}
  \put(375.75,60.66){\fontsize{6.47}{7.76}\selectfont $F_5$}
  \end{picture}%
\fi
\caption{\label{pic1}%
 }
\end{figure}

\begin{lemma}\label{L:Interpolation modified band}
There exists $\delta_0>0$ such that for $\delta<\delta_0$ the following statement holds. Let $f:[0,\pi]\to \mc N$, $g:[0,\pi]\to \mc N$ be two $C^0$ maps such that $\Vert f-g\Vert_{L^\infty}\leq\delta$, $\int\vert f'\vert^2\leq\delta'$, $\int\vert g'\vert^2\leq\delta'$. Then there exists $\rho\in (0,1/2]$ and a map $v: \textit{MB}_{1-\rho, 1}\to \mc N$ such that $v(1-\rho,\theta)=f(\theta)$, $v(1,\theta)=g(\theta)$, and $E(v)\leq C\delta^{1/2}\delta'^{1/2}$ for some constant $C$ depending only on $\mc N,\rho$. Moreover, for the arc of the removed half disks of the modified band (reparametrized by angle parameter $\phi$), we have 
\[\int_0^\pi\vert v'(\phi)\vert^2 ds(\phi)\leq C\delta^2 ,\]
where $ds(\phi)$ is the intrinsic arc length integral.
\end{lemma}

\begin{proof}
We construct a change of variables. Let us construct a map $h:D^+_{1}\setminus D^+_{1-\rho}\to \textit{MB}_{1-\rho,1}$ as follows: given $s\in[1-\rho, \rho]$, we consider the arc $\textit{MB}_{1-\rho,1}\cap \{r\equiv s\}$, and denote the length of this arc by $l_s$; for any fixed $(r,\theta)\in D^+_{1}\setminus D^+_{1-\rho}$, $h$ maps it to $(r,\theta')$, where the ratios satisfy
\[ [\vert \theta-\pi/2\vert : \pi/2]= [\vert \theta'-\pi/2\vert: l_r].   \]
Then the Jacobian and the gradient of this change of variables is bounded by some universal constant.  Let $\tilde v$ be given in Lemma \ref{L:Interpolation band}, then $v=\tilde v\circ (h^{-1})$ is the desired interpolation map. The only thing we need to check is the last estimate. In particular, the energy of $v$ on the boundary of the removed half disks is bounded by that of $\tilde v$ on $\theta=0$ times a universal constant, and the energy of $\tilde v$ therein is bounded .
\end{proof}

Now we start proving that the harmonic replacement $u\to w$ is continuous as a map from $C^0\cap W^{1,2}$ to $W^{1,2}$ for $u$ with small energy.

\begin{proof}[Proof of $W^{1,2}$-continuity]
We prove by contradiction. Suppose the $W^{1,2}$-continuity of the harmonic replacement process fails, then we can find a sequence of maps $\{u_i:D^+\to \mc N\}$ converging to $u_\infty$ in $C^0(\overline{D^+}, \cN) \cap W^{1,2}(D^+, \cN)$, but their free boundary harmonic replacements $\{w_i\}$ does not converge to the corresponding replacement $w_\infty$ (of $u_\infty$) in $W^{1,2}$. Note that a subsequence of $w_i$ must converge weakly to $w_\infty$, so by lower-semicontinuity of energy $E(w_\infty)\leq\liminf E(w_i)$.
\vskip 1.5mm

By the convexity of energy for free boundary harmonic maps (Theorem \ref{T:convexity2}), the fact that $w_i$ does not converge to $w_\infty$ in $W^{1,2}$ implies that there exists $\eps>0$ and a subsequence (still denoted by $w_i$) such that $ E(w_i)-E(w_\infty) \geq \eps$.\vskip 1.5mm

We divide $D^+$ into five different pieces $D^+=\cup_{j=1}^5 F_j$:
\begin{enumerate}
\item $F_1=\{(r,\theta):r\in[0, \eta\lambda(1-\rho)]\}$,
\item $F_2=\{(r,\theta):r\in[\eta\lambda(1-\rho),\lambda(1-\rho)]\}$,
\item $F_3=\textit{MB}_{\lambda(1-\rho),\lambda}$,
\item $F_4=\{y\in D^+:\vert y-(\lambda-\lambda\rho/2,0)\vert\leq \lambda\rho/2\}\cup\{y\in D^+:\vert y-(\lambda-\lambda\rho/2,\pi)\vert\leq \lambda\rho/2\}$,
\item $F_5=\{(r,\theta):r\in[\lambda,1]\}$,
\end{enumerate} 
see Figure \ref{pic1}. Here $\eta,\rho,\lambda\in (0 ,1)$ will be determined later. Note that the $F_j$'s have some common boundaries, and we will see in the construction below that the maps on the common boundaries share the same value. Let us choose a radius $r_0>1/2$ such that $\int_{D^+\setminus D^+_{1-r_0}}\vert\nabla u_\infty\vert^2\leq \eps/16$. Then given $i$ large such that $\Vert u_i-u_\infty\Vert^2_{W^{1,2}}\leq\eps/16$, we have $\int_{D^+\setminus D^+_{1-r_0}}\vert\nabla u_\infty\vert^2\leq \eps/8$. Using the co-area formula, we can pick some radius $r_1\in (r_0, 1)$ such that $\int_{r=r_1}\vert\partial_\theta u_\infty\vert\leq C(r_0)\eps$ and $\int_{r=r_1}\vert\partial_\theta u_i\vert\leq C(r_0)\eps$.\vskip 0.5em

\noindent{\bf Construction of $v_i$}: 
\begin{itemize}
\item On $F_1$, we define $v$ to be $w_\infty$ after rescaling to a suitable scale to fit $F_1$;

\item On $F_2$, we define $v$ to be the part of $u_\infty$ defined on $D^+_{1}\setminus D^+_{1-r_1}$ after inversion (in polar coordinate, we change $r$ to be $2-r_1-r$) and rescaling. Note that the energy $v$ on $F_2$ is bounded by the energy of $u_\infty$ on $D^+_{1}\setminus D^+_{1-r_1}$ times a universal constant $4$;

\item On $F_3$, we define $v$ to be the interpolation between $u_\infty\vert_{\{r=r_1\}}$ and $u_i\vert_{\{r=r_1\}}$ from Lemma \ref{L:Interpolation modified band} on the modified band, then rescaling to fit $F_3$;

\item On $F_4$, we define $v$ to be the map constructed in Lemma \ref{L:small free boundary map}, which provides two maps with small energy;

\item On $F_5$, we define $v$ to be the the part of $u_i$ defined on $D^+_{1}\setminus D^+_{1-r_1}$.
\end{itemize}

Here $\rho$ comes from Lemma \ref{L:Interpolation modified band}, $\eta=1-r_1$ and $\lambda=1-r_1$. These parameters are chosen in order to guarantee that in the definition of $v$'s the rescaling are possible. i.e. the ratio between inner radius and outer radius of the bands or modified bands do not change.
\vskip 1.5mm

Now we claim some properties of $v$ we constructed. First, on $\partial^A$ the fixed boundary, $v_i=u_i$; on $\partial^C$ the free boundary, $v_i$ always has image in $\Gamma$. Second, $v_i$ is continuous. So it is an eligible comparison map (with $w_\infty$).\vskip 1.5mm

Finally we estimate the energy of $v_i$. We will repeatedly use (without mentioning) the fact that the harmonic energy is invariant under conformal reparametrization on the domain. Note that the energy of $v_i$ on $F_1$ equals $E(w_\infty)$; on $F_2$ and $F_5$ the total energy of $v_i$ is bounded by $\eps/3$; on $F_3$ by Lemma \ref{L:Interpolation modified band} the energy of $v_i$ is bounded by a constant times $\Vert u_\infty-u_i\Vert_{C^0}^{1/2}$, and on $F_4$ by Lemma \ref{L:small free boundary map} the energy of $v_i$ is bounded by a constant times $\Vert u_\infty-u_i\Vert_{C^0}^2$. Combining all the pieces gives $E(v_i)\leq E(w_\infty)+\eps/3+C\Vert u_\infty-u_i\Vert_{C^0}^{1/2}$. Since $u_i$ converges to $u_\infty$ in $C^0$, when $i$ large enough $C\Vert u_\infty-u_i\Vert_{C^0}^{1/2}\leq\eps/3$. In conclusion, when $i$ is large enough, $E(v_i)\leq E(w_\infty)+2\eps/3\leq\liminf E(w_n)-\eps/3< E(w_i)$.\vskip 1.5mm

However $v_i$ shares the same fixed boundary with $w_i$. So we construct a map with same fixed boundary with $w_i$ along $\partial^A$, and $v_i(\partial^C)\subset\Gamma$, but has energy strictly less than that of $w_i$; this is a contradiction to the energy minimizing property of harmonic replacement (Theorem \ref{T:convexity2}). Therefore the harmonic replacement is continuous in $W^{1,2}$.
\end{proof}

\subsubsection{$C^0$-continuity}
\begin{proof}[Proof of $C^0$-continuity]
Now we are ready to prove the $C^0$-continuity of the harmonic replacement process. We argue by contradiction. If the harmonic replacement process is not $C^0$, we can find a sequence of maps $u_i:D^+\to \mc N$ converge to $u_\infty$ in $C^0(\overline{D^+})\cap W^{1,2}(D^+)$ but the corresponding free boundary harmonic replacements $\Vert w_i-w_\infty\Vert_{C^0(\overline{D^+})}\geq\eps>0$.\vskip 1.5mm

By the gradient estimate of free boundary harmonic maps (Theorem \ref{e-reg}), we know that $w_i$'s are equicontinuous on any compact subset of $D^+$ away from the fixed boundary $\partial^A$. Moreover, by Qing's result in \cite{Qing93}, we know that $w_i$'s are equicontinuous on any compact subset of $\overline{D^+}$ besides the two corner points $(1, 0)$ and $(1, \pi)$. Finally, since we have the $W^{1,2}$-continuity of the harmonic replacement process, $\Vert w_i-w_{\infty}\Vert_{W^{1,2}(D^+)}\to 0$ as $i\to \infty$. Then by checking the proof of Theorem \ref{T: Continuity of free boundary harmonic map}, we can see that $w_i$'s are also equicontinuous at the corner points. Thus by the Arzela-Ascoli theorem, up to a subsequence $w_i$ must converge to some $w_\infty'$ in $C^0(\overline {D^+})$. The $W^{1, 2}$-convergence of $w_i$ then implies $w_\infty=w_\infty'$, but this contradicts $\Vert w_i-w_\infty\Vert_{C^0(\overline {D^+})}\geq\eps>0$. So we finish the proof.
\end{proof}

\subsection{Uniform continuity of energy improvement}

In this part, we prove two inequalities regarding energy improvements for two sets of free boundary harmonic replacements on generalized balls (see Definition \ref{D:generalized balls}). These inequalities play a key role in the proof of Theorem \ref{T:Tightening theorem}. Similar results were proved by Colding-Minicozzi \cite[Section 3]{CM08} for fixed boundary harmonic replacements. For first reading, we suggest the readers coming back to the proof after reading Section \ref{SS:tightening process}. \vskip 1.5mm

Let $\eps_0$ be as in Theorem \ref{T:convexity1} and Theorem \ref{T:convexity2}. We adopt the following notations: given a map $u \in C^0(\overline{D}, \cN) \cap W^{1,2}(D, \cN)$ with $u(\partial D)\subset \Gamma$ and a finite collection $\mc B$ of disjoint generalized closed balls in $D$ so that the energy of $u$ on $\cup_{\mc B}B$ is at most $\epsilon_0/3$, let $H(u,\mc B): D\to \mc N$ denote the map that coincides with $u$ on $D\setminus\cup_{\mc B}B$ and is equal to the free boundary harmonic replacements of $u$ on $\cup_{\mc B}B$; (see Remark \ref{Re:free boundary harmonic replacement}, 
where the existence is guaranteed by Theorem \ref{T:existence and regularity of free boundary harmonic replacement}). Note we also have $H(u, \mc B)(\partial D)\subset \Gamma$, and we will call $H(u,\mc B)$ the {\em free boundary harmonic replacement of $u$ on $\mc B$}. 
Given two such disjoint collections $\mc B_1, \mc B_2$, we use $H(u,\mc B_1,\mc B_2)$ to denote $H(H(u,\mc B_1),\mc B_2)$. Recall that for $\alpha\in (0,1]$, $\alpha \mc B$ will denote the collection of concentric balls with radii that are shrunk by the factor $\alpha$ in the sense of Definition \ref{D:generalized balls}.\vskip 1.5mm

First we have the following interpolation formula for free boundary replacements, c.f. \cite[Lemma 3.11]{CM08}. Denote $\kappa$ as the radius such that the Fermi coordinates system exists in a tubular neighborhood of radius $\kappa$ surrounding $\Gamma$ where Lemma \ref{L:Fermi equivalent to Rn} in the appendix applies.

\begin{lemma}
\label{L: lem311}
There exists $\tau>0$ so that given $f, g:\partial_R^A\to \cN$ two $C^0\cap W^{1,2}$ maps with $f(0)$, $f(\pi)$, $g(0)$, $g(\pi)\in\Gamma$, if $f$ and $g$ agree at one point on $\partial_R^A$ and satisfy
\[R\int_{\partial_R^A}\vert f'-g'\vert^2\leq\tau^2,\]
and
\[\vert f(\theta)-f(0)\vert\leq\kappa/3\mbox{ for all $0\leq \theta\leq\pi$},\]
then there exist $\rho\in(0,R/2]$ and a $C^0\cap W^{1,2}$-map $w:D^+_R\setminus D^+_{R-\rho}\to M$ so that
\[w(R-\rho,\theta)=f(\theta)\,\mbox{ and }\, w(R,\theta)=g(\theta),\]
and the image of $w\vert_{\partial_R^C\setminus \partial_{R-\rho}^C}$ lies in $\Gamma$, and 
the following estimate holds
\[\int_{D^+_R\setminus D^+_{R-\rho}}\vert\nabla w\vert^2\leq C(R\int_{\partial^A_R}\vert f'\vert^2+\vert g'\vert^2)^{1/2}(R\int_{\partial^A_R}\vert f'-g'\vert^2)^{1/2},\]
where $C>0$ is a universal constant depending only on $\kappa, \tau$.
\end{lemma}

\begin{rem}
This lemma generalizes an interpolation formula by Colding-Minicozzi for maps defined on circles \cite[Lemma 3.11]{CM08} to maps defined on half circles, and the constructed interpolating map has the chord boundary $\partial_R^C\setminus \partial_{R-\rho}^C$ lying on $\Gamma$. 
In \cite[Lemma 3.11]{CM08}, they first took the linear interpolation of $f,g$ in $\mb R^N$, and then projected it back to the ambient manifold $\cN$. However this method doesn't work in the free boundary case because the projection $\mb R^N\to \cN$ may not map the boundary $\partial_R^C$ to the constraint submanifold $\Gamma$.

Here we use Fermi coordinate system to construct the desired interpolation between $f$ and $g$. In the proof we have to work in two different coordinate systems, and Lemma \ref{L:Fermi equivalent to Rn} will be used to show the equivalence.
\end{rem}

\begin{proof}[Proof of Lemma \ref{L: lem311}]
Note that there is one point $0\leq\theta\leq\pi$ such that $f(\theta)=g(\theta)$. Choose $\tau=\kappa/3\pi$. Then by the assumptions and taking integration, we get $\vert f(\theta)-g(\theta)\vert\leq \kappa/2$ for all $\theta\in[0, \pi]$, so that $\vert g(\theta)-g(0)\vert\leq\kappa$ for all $\theta\in[0, \pi]$.

As a result, we can assume the images of $f,g$ all lie in a convex neighborhood $U_\kappa$ of $f(0)$, and in $U$ we can pick the Fermi coordinate systems $\{y^1, \cdots, y^n\}$ as in Lemma \ref{L:Fermi equivalent to Rn}, where $\Gamma$ is a subset of $\{y^{k+1}=\cdots=y^n=0\}$. Recall in Lemma \ref{L:Fermi equivalent to Rn}, $g^1, g^2, g^3$ denote respectively the metric of $\mc N$, the flat metric in $\{y^1,\cdots, y^n\}$ and the flat metric of $\mb R^N$. We will use $\vert\cdot\vert$ to denote the norm under the metric $g^3$.

Since the statement is scaling invariant, it suffices to prove the case $R=1$. For $\rho\leq 1/2$ to be determine, define $w:D^+\setminus D^+_{1-\rho}\to\mb{R}^N$ by
\[w(r,\theta)=f(\theta)+\left(\frac{r+\rho-1}{\rho}\right)(g(\theta)-f(\theta))\,.\]
Note on $\partial_1^C\setminus \partial_{1-\rho}^C$, since $f(0),f(\pi),g(0),g(\pi)\in\Gamma$, in the Fermi coordinate chart $f(0),f(\pi),g(0),g(\pi)$ lie in the plane $\{y^{k+1}=\cdots=y^n=0\}$, then as a result, the interpolation function $w$ also has the image $w\vert_{\partial_1\setminus \partial_{1-\rho}}$ lying in the same plane, and it turns out that the image of $w\vert_{\partial_1\setminus \partial_{1-\rho}}$ lies in $\Gamma$. 

The energy density in this coordinate system is
\[e(w)=\sum_{k=1}^2\sum_{i,j=1}^n g^1_{ij}\frac{\partial w^i}{\partial x_k}\frac{\partial w^j}{\partial y_k}\leq(1+\alpha)\Vert\nabla w\Vert_{g^2}^2\,.\]

Now we proceed to prove the estimate of $\Vert\nabla w\Vert_{g^2}^2$. First we need a Wirtinger type inequality. Suppose $h(\cdot)$ is a function on $\partial^A$ and $h(s)=0$ for some $s\in[0,\pi]$, then for any $t\in[0,\pi]$, we have
\[\vert h(t)\vert=\vert h(t)-h(s)\vert\leq\int_{s}^t\vert h'(x)\vert dx\leq\int_{\partial^A_1}\vert h'\vert\,.\]
Integrating the square of both side for $t\in[0,\pi]$ we get
\[\int_{\partial^A_1}\vert h(t)\vert^2\leq \pi (\int_{\partial^A_1}\vert h'\vert)^2\leq \pi^2\int_{\partial^A_1}\vert h'\vert^2\,.\]
Then we get
\[\int_{\partial^A_1}\vert f-g\vert^2\leq \pi^2 \int_{\partial^A_1}\vert f'-g'\vert^2\,.\]

Thus
\begin{equation*}
\begin{split}
E(D^+\setminus D^+_{1-\rho})&\leq\int_{D^+\setminus D^+_{1-\rho}}(1+\alpha)\Vert\nabla w\Vert_{g^2}^2\\
&\leq (1+\alpha)\int_{1-\rho}^1\left(\frac{1}{\rho^2}\int_{0}^\pi\Vert f-g\Vert_{g^2}^2(\theta)d\theta+\frac{1}{r^2}\int_0^\pi(\Vert f'\Vert_{g^2}^2+\Vert g'\Vert_{g^2}^2)(\theta)d\theta\right)rdr\\
&\leq (1+\alpha)^2\int_{1-\rho}^1\left(\frac{1}{\rho^2}\int_{0}^\pi\vert f-g\vert^2(\theta)d\theta+\frac{1}{r^2}\int_0^\pi(\vert f'\vert^2+\vert g'\vert^2)(\theta)d\theta\right)rdr\\
&\leq (1+\alpha)^2\frac{8}{\rho}\int_0^\pi\vert f'-g'\vert^2(\theta)d\theta+2\rho\int_0^\pi(\vert f'\vert^2+\vert g'\vert^2)(\theta)d\theta\\
&\leq8\frac{(1+\alpha)^2}{(1-\alpha)}(\int_{\partial^A_1}\vert f'\vert^2+\vert g'\vert^2)^{1/2}(\int_{\partial^A_1}\vert f'-g'\vert^2)^{1/2}\,,
\end{split}
\end{equation*}
once we pick $\rho^2=(\int_{\partial^A_1}\vert f'-g'\vert^2)^{1/2}/(\int_{\partial^A_1}\vert f'\vert^2+\vert g'\vert^2)^{1/2}$.
\end{proof}

Now we are ready to prove the energy improvement inequalities for free boundary harmonic replacements. Similar results were first obtained by Colding-Minicozzi for fixed boundary harmonic replacements \cite[Lemma 3.8]{CM08}.

\begin{lemma}
\label{lem3.8}
There is a constant $k>0$ so that if $u:D\to \cN$ is in $C^0\cap W^{1,2}$ and $\mc{B}_1,\mc{B}_2$ are two finite collections of disjoint closed generalized balls in $D$ so that the energy of $u$ on each $\cup_{\mc B_i}B$ is at most $\epsilon_0/3$, then
\begin{equation}
    E(u)-E(H(u,\mc B_1,\mc B_2))\geq k(E(u)-E(H(u,\frac{1}{2}\mc B_2)))^2.
\end{equation}
Furthermore, for any $\mu\in[1/8,1/2]$ we have
\begin{equation}
\label{E:lem3.8 second formula}
    \frac{E(u)-E(H(e,\mc B_1))^{1/2}}{k}+E(u)-E(H(u,2\mu\mc B_2))\geq E(H(u,\mc B_1))-E(H(u,\mc B_1,\mu\mc B_2)).
\end{equation}
\end{lemma}

\begin{proof}
The proof is analogous to the proof of \cite[Lemma 3.8]{CM08}, and the main difficulty arises from those boundary balls. We will include the details for completeness, and focus on how we use the new interpolation result, i.e. Lemma \ref{L: lem311}.\vskip 1.5mm

Let $\mc B_1=\{B_\alpha^1\}$ and $\mc B_2=\{B_j^2\}$. We need to clarify the second replacement. Observe that the total energy of $u$ on the union of the balls $\mc B_1\cup\mc B_2$ is at most $2\eps_0/3$, and the free boundary harmonic replacement on $\mc B_1$ does not change the map outside these balls and is energy non-increasing, then it follows that the total energy of $H(u,\mc B_1)$ on $\mc B_2$ is at most $2\eps_0/3$.\vskip 1.5mm

We will divide $\mc B_2$ into two disjoint subsets $\mc B_{2,+}$ and $\mc B_{2,-}$, set
\[\mc B_{2,+}=\{B_j^2\in\mc B_2:\frac{1}{2}B_j^2\subset B_\alpha^1 \mbox{ for some } B_\alpha^1\in \mc B_1\}\mbox{ and }\mc B_{2,-}=\mc B_2\setminus\mc B_{2,+}. \]
Since the balls in $\mc B_2$ are disjoint, it follows that
\[E(u)-E(H(u,\frac{1}{2}\mc B_2))=\left(E(u)-E(H(u,\frac{1}{2}\mc B_{2,+}))\right)+\left(E(u)-E(H(u,\frac{1}{2}\mc B_{2,-}))\right). \]

Now we have two cases.

\vspace{1em}
{\bf Case 1:} Suppose
\[ \left(E(u)-E(H(u,\frac{1}{2}\mc B_{2,+}))\right)\geq\left(E(u)-E(H(u,\frac{1}{2}\mc B_{2,-}))\right), \]
then
\[ \left(E(u)-E(H(u,\frac{1}{2}\mc B_{2,+}))\right)\geq\frac{1}{2}\left(E(u)-E(H(u,\frac{1}{2}\mc B_{2}))\right). \]

Since the balls in $\mc B_{2,+}$ are contained in balls in $\mc B_1$ and harmonic replacement minimize energy, we get
\[ E(u,\mc B_1,\mc B_2)\leq E(u,\mc B_1)\leq E(u,\mc B_{2,+}). \]
so that
\begin{equation*}
    \begin{split}
        E(u)-E(H(u,\mc B_1,\mc B_2))&\geq E(u)-E(H(u,\frac{1}{2}\mc B_{2,+}))\\
        &\geq \frac{1}{2}\left( E(u)-E(H(u,\frac{1}{2}\mc B_2))\right)\\
        &\geq k\left( E(u)-E(H(u,\frac{1}{2}\mc B_2))\right)^2\,,
    \end{split}
\end{equation*}
if $k\leq 1/(2\eps)$.

\vspace{1em}
{\bf Case 2:} Suppose 
\[\left(E(u)-E(H(u,\frac{1}{2}\mc B_{2,+}))\right)\leq\left(E(u)-E(H(u,\frac{1}{2}\mc B_{2,-}))\right), \]
then
\[\left(E(u)-E(H(u,\frac{1}{2}\mc B_{2,-}))\right)\geq\frac{1}{2}\left(E(u)-E(H(u,\frac{1}{2}\mc B_{2}))\right). \]

Let $\tau>0$ be given by Lemma \ref{L: lem311}. We can assume that
\begin{equation}
\label{E:assumption on energy decay over B1}
    9\int_D\vert \nabla H(u,\mc B_1)-\nabla u\vert^2\leq \tau^2.
\end{equation}
Otherwise Theorem \ref{T:convexity1} and Theorem \ref{T:convexity2} giveg the result with $k=\tau^2/(2\eps_0^2)$. In fact, if 
\[    9\int_D\vert \nabla H(u,\mc B_1)-\nabla u\vert^2>\tau^2, \]
then applying Theorem \ref{T:convexity1} and Theorem \ref{T:convexity2} to each classical ball or boundary ball in $\mc B_1$, we get
\[ E(u)-E(H(u,\mc B_1,\mc B_2))\geq E(u)-E(H(u,\mc B_1))\geq\frac{1}{2}\int_D\vert\nabla u-\nabla H(u,\mc B_1)\vert^2>\tau^2/18. \]
Noting that $E(u)-E(H(u,\frac{1}{2}\mc B_2))\leq (1/3)\eps_0$, we get the desired estimate.\vskip 1.5mm

Assuming (\ref{E:assumption on energy decay over B1}), we want to show the following estimate for each $B_j^2\in\mc B_{2,-}$ that

\begin{equation}\label{3.16}
\begin{split}
        \int_{B_j^2}\vert\nabla H(u,\mc B_1)\vert^2-\int_{B_j^2}\vert\nabla H(u,\mc B_1,B_j^2)\vert^2\geq \int_{\frac{1}{2}B_j^2}\vert\nabla u\vert^2-\int_{\frac{1}{2}B_j^2}\vert\nabla H(u,\frac{1}{2}B_j^2)\vert^2\\- C\left(\int_{B_j^2}\vert\nabla u\vert^2+\vert\nabla H(u,\mc B_1)\vert^2\right)^{\frac{1}{2}}\left(\int_{B_j^2}\vert\nabla(u-H(u,\mc B_1))\vert^2\right)^{\frac{1}{2}}.
\end{split}
\end{equation}

If this estimate is true, summing it over balls in $\mc B_{2,-}$ and using Cauchy-Schwartz inequality for discrete sums $\vert\sum a_jb_j\vert\leq(\sum a_j^2)^{1/2}(\sum b_j^2)^{1/2}$, Theorem \ref{T:convexity1} and Theorem \ref{T:convexity2}, we get
\begin{align}
        E(H(u,\mc B_1))-E(H(u,\mc B_1,\mc B_{2,-}))&\geq E(u)-E(H(u,\frac{1}{2}\mc B_{2,-}))-C\eps_0^{1/2}(E(u)-E(H,\mc B_1))^{1/2}\notag\\
        &\geq E(u)-E(H(u,\frac{1}{2}\mc B_{2,-}))-C\eps_0^{1/2}(E(u)-E(H,\mc B_1, \mc B_{2,-}))^{1/2}.
\end{align}
Then noting that $E(H(u,\mc B_1))-E(H(u,\mc B_1,\mc B_{2,-}))\leq E(u)-E(H(u,\mc B_1,\mc B_{2,-}))\leq E(u)-E(H(u,\mc B_1,\mc B_{2}))\leq 2\eps_0/3<\eps_0$, we get the desired estimate.

So we only need to prove (\ref{3.16}) to conclude case 2. If $B_j^2$ is a classical ball, this is just the result in \cite{CM08}. Now let us consider the case when $B_j^2$ is the boundary ball $D^+_R$ of radius $R$ centered at $0$ in the upper half plane $\mb H^2$. Set $u_1=H(u,\mc B_1)$. By the co-area formula, there exists some $r\in[3R/4,R]$ with
\[\int_{\partial^A_r}\vert\nabla u_1-\nabla u\vert^2\leq\frac{9}{R}\int_{3R/4}^R(\int_{\partial^A_s}\vert\nabla u_1-\nabla u\vert^2)ds\leq\frac{9}{r}\int_{D^+_R}\vert\nabla u_1-\nabla u\vert^2\,,\]
\[\int_{\partial^A_r}(\vert \nabla u_1\vert^2+\vert\nabla u\vert^2)\leq\frac{9}{R}\int_{3R/4}^R(\int_{\partial^A_s}\vert\nabla u_1\vert^2+\vert\nabla u\vert^2)ds\leq\frac{9}{r}\int_{D^+_R}(\vert\nabla u_1\vert^2+\vert\nabla u\vert^2)\leq \frac{24\eps_0}{R}\,.\]
Note the second estimate indicates that the length of the image of $u\vert_{\partial^A_R}$ is bounded by a universal constant times $\eps_0$. So if $\eps_0$ is smaller than a constant multiple of $\kappa$ (where $\kappa$ is the radius for Fermi coordinates of $\Gamma$; see Lemma \ref{L: lem311}), we will get $\vert u(x)-u(0)\vert\leq \kappa/2$ for all $x\in \partial^A_R$. Then we can apply Lemma \ref{L: lem311} to get some $\rho\in(0,r/2]$ and a map $w:D^+_r\setminus D^+_{r-\rho}\to \cN$ with $w(r,\theta)=u_1(r,\theta)$ and $w(r-\rho,\theta)=u(r,\theta)$, such that

\[\int_{D^+_r\setminus D^+_{r-\rho}}\vert \nabla w\vert^2\leq C(\int_{D^+_R}\vert\nabla u\vert^2+\vert\nabla H(u,\mc B_1)\vert^2)^{1/2}(\int_{D^+_R}\vert\nabla(u-H(u,\mc B_1))\vert^2)^{1/2}\,.\]

The map $x\to H(u, D^+_r)(rx/(r-\rho))$ maps $D^+_{r-\rho}$ to $\cN$ and agrees with $w$ on $\partial_{r-\rho}^A$. So we get a map from $D^+_R$ to $\cN$ which is equal to $H(u,\mc B_1)$ on $D^+_R\setminus D^+_r$, equal to $w$ on $D^+_r\setminus D^+_{r-\rho}$ and is equal to $H(u,D^+_r)(r\cdot/(r-\rho))$ on $D^+_{r-\rho}$. This new map gives an upper bound for the energy of $H(u_1,D^+_R)$:
\[ \int_{D^+_R}\vert\nabla H(u_1,D^+_R)\vert^2\leq \int_{D^+_R\setminus D^+_r}\vert\nabla u_1\vert^2+\int_{D^+_r\setminus D^+_{r-\rho}}\vert\nabla w\vert^2+\int_{D^+_r}\vert\nabla H(u,B_r)\vert^2. \]

Using previous estimate and the fact $\vert\vert\nabla u_1\vert^2-\vert\nabla u\vert^2\vert\leq(\vert\nabla u_1\vert+\vert \nabla u\vert)\vert \nabla(u-u_1)\vert$ we get

\begin{equation*}
    \begin{split}
        \int_{D^+_R}\vert\nabla u_1\vert^2-\int_{D^+_R}\vert\nabla H(u_1,D^+_R)\vert^2\geq &\int_{D^+_r}\vert\nabla u_1\vert^2-\int_{D^+_r}\vert\nabla H(u,D^+_r)\vert^2-\int_{D^+_r\setminus D^+_{r-\rho}}\vert\nabla w\vert^2\\
        \geq &\int_{D^+_r}\vert \nabla u\vert^2-\int_{D^+_r}\vert\nabla H(u,D^+_r)\vert^2-\\
        &C\left(\int_{D^+_r}\vert\nabla u\vert^2+\vert\nabla u_1\vert^2\right)^{1/2}\left(\int_{D^+_r}\vert\nabla(u-u_1)\vert^2\right)^{1/2}.
    \end{split}
\end{equation*}

Since $r>R/2$, we have
\begin{equation*}
\begin{split}
    \int_{D^+_{R/2}}\vert \nabla u\vert^2&= \int_{D^+_r}\vert\nabla u\vert^2-\int_{D^+_r\setminus D^+_{R/2}}\vert\nabla u\vert^2\leq \int_{D^+_r}\vert\nabla u\vert^2-\int_{D^+_r\setminus D^+_{R/2}}\vert\nabla H(u,D^+_r)\vert^2\\
    &\leq \int_{D^+_r}\vert\nabla u\vert^2-\int_{D^+_r}\vert\nabla H(u,D^+_r)\vert^2+\int_{D^+_{R/2}}\vert\nabla H(u,D^+_r)\vert^2\\
    &\leq \int_{D^+_r}\vert\nabla u\vert^2-\int_{D^+_r}\vert\nabla H(u,D^+_r)\vert^2+\int_{D^+_{R/2}}\vert\nabla H(u,D^+_{R/2})\vert^2.
\end{split}
\end{equation*}
Combining this with the previous estimate we get (\ref{3.16}). Hence we conclude the first formula in the lemma.
\vskip 1.5mm

For the second formula, by the same argument as in \cite{CM08}, based on the proof of first formula as above we get the second formula.  Then we complete the proof.
\end{proof}

\subsection{Tightening process}
\label{SS:tightening process}
Now we have developed the tools to prove Theorem \ref{T:Tightening theorem}. The proof is actually the same as the proof of \cite[Theorem 2.1]{CM08}, because every ingredient in the proof has been proved to be true for generalized balls in previous sections. We give the proof here for the convenience of the readers.\vskip 1.5mm

Given a sweep-out $\sigma\in\Omega$ and $\eps\leq \eps_0$, we define the maximal improvement for free boundary harmonic replacement on families of generalized balls with energy at most $\eps$ by
\[ e_{\sigma,\eps}(t)=\sup_{\mc B}\{E(\sigma(\cdot,t))-E(H(\sigma(\cdot,t),\frac{1}{2}\mc B))\}. \]
Here the supremum is taken over all collection $\mc B$ of disjoint closed generalized balls where the total energy of $\sigma(\cdot,t)$ on $\mc B$ is at most $\eps$. $e_{\sigma,\eps}(t)$ is positive if $\sigma(\cdot,t)$ is not harmonic. \vskip 1.5mm

We first show that the maximal improvement of a given slice (which is not harmonic) can control the maximal improvement of any nearby slices.

\begin{lemma}\label{lem3.20}
Given $t\in(0, 1)$, if $\sigma(\cdot,t)$ is not harmonic and $\eps<\eps_0$, then there is an open interval $I^t$ containing $t$ so that $e_{\sigma,\eps/2}(s)\leq 2e_{\sigma,\eps}(t)$ for all $s$ in the double interval $2I^t$.
\end{lemma}

\begin{proof}
By Theorem \ref{T:continuity of harmonic replacement}, there exists $\delta_1>0$ depending on $t$ such that if $\Vert\sigma(\cdot,t)-\sigma(\cdot,s)\Vert_{C^0\cap W^{1,2}}<\delta_1$, and $\mc B$ is a finite collection of disjoint generalized balls where both $\sigma(\cdot,t)$ and $\sigma(\cdot,s)$ have energy at most $\eps_0$, then
\begin{equation}
\label{3.22}
\vert E(H(\sigma(\cdot,s),\frac{1}{2}\mc B))-E(H(\sigma(\cdot,t),\frac{1}{2}\mc B))\vert\leq \frac{1}{2}e_{\sigma,\eps}(t). 
\end{equation}

Since $t\to \sigma(\cdot,t)$ is continuous as a map to $C^0\cap W^{1,2}$, we can choose $I^t$ such that for $s\in I^{2t}$, $\Vert\sigma(\cdot,t)-\sigma(\cdot,s)\Vert_{C^0\cap W^{1,2}}<\delta_1$ holds and 
\begin{equation}
\label{3.23}
\int_{D}\vert\vert\nabla\sigma(\cdot,t)\vert^2-\vert\nabla\sigma(\cdot,s)\vert^2\vert\leq\min\{\eps,e_{\sigma,\eps}(t)\}.
\end{equation}
Now suppose $s\in 2I^t$ and the energy of $\sigma(\cdot,s)$ is at most $\eps/2$ on a collection $\mc B$. It follows from (\ref{3.23}) that the energy of $\sigma(\cdot, t)$ on $\mc B$ is at most $\eps$. Combining (\ref{3.22}) we get
\[ \vert E(\sigma(\dot,s))-E(H(\sigma(\cdot,s)\frac{1}{2}\mc B))-E(\sigma(\cdot,t))+E(H(\sigma(\cdot,t),\frac{1}{2}\mc B))\vert\leq e_{\sigma,\eps}(t). \]

Since this estimate applies to any $\mc B$, also notice that $e_{\sigma,\eps}$ is monotone increasing in $\eps$, we get that $e_{\sigma,\eps/2}(s)\leq 2e_{\sigma,\eps}(t)$.
\end{proof}

Theorem \ref{T:Tightening theorem} indicates that our tightening process should effectively decrease the energy of those non-harmonic slices with large energy (when $E>W/2$). The next lemma shows that we can find a replacement to decrease the energy of those slices for certain amount.

\begin{lemma}\label{lem3.24}
If $\tilde \gamma\in \Omega$ has no harmonic slices other than $u(\cdot,0)$ and $u(\cdot,1)$, then we get an integer $m$ depending on $\tilde \gamma$, $m$ collections of generalized balls $\mc B_1,\cdots,\mc B_m$ in $D$ where the balls in each collection $\mc B_j$ are pairwise disjoint, and $m$ continuous functions $r_1,\cdots,r_m:[0,1]\to [0,1]$ so that for each $t$: 
\begin{enumerate}
    \item at most two $r_j(t)$'s are positive and $\frac{1}{2}\int_{r_j(t)\mc B}\vert \nabla\tilde\gamma(\cdot,t)\vert^2<\eps_0/3$ for each $j$;
    \item if $E(\tilde\gamma(\cdot,t))>W/2$, then there exists $j(t)$ so that the harmonic replacement on $(r_j(t)/2)\mc B_{j(t)}$ decreases energy by at least $e_{\tilde\gamma,\eps_0/8}(t)/8$.
\end{enumerate}
\end{lemma}

\begin{proof}
Since the energy of the slices is continuous in $t$, the set $I=\{t:E(\tilde\gamma(\cdot,t))\geq W/2\}$ is compact. For each $t\in I$, choose a finite collection $\mc B^t$ of disjoint closed balls in $D$ with $\frac{1}{2}\int_{\mc B^t}\vert\nabla\tilde\gamma(\cdot,t)\vert^2\leq\eps_0/4$ so that
\[ E(\tilde\gamma(\cdot,t))-E(H(\tilde\gamma(\cdot,t),\frac{1}{2}\mc B^t))\geq\frac{1}{2}e_{\tilde\gamma,\eps_0/4}(t). \]
Note $u(\cdot,0)$ and $u(\cdot,1)$ does not lie in $I$, this energy improvement is positive. Lemma \ref{lem3.20} gives an open interval $I^t$ containing $t$ so that for all $s\in 2I^t$,
\[ e_{\tilde \gamma,\eps_0/8}(s)\leq2e_{\tilde\gamma,\eps_0/4}(t). \]

Using the continuity of $\tilde\gamma(\cdot,s)$ in $C^0\cap W^{1,2}$, we can shrink $I^t$ so that $\tilde\gamma(\cdot,s)$ has energy at most $\eps_0/3$ in $\mc B^t$ for $s\in 2I^t$ and in addition
\begin{equation}
    \vert E(\tilde\gamma(\cdot,s))-E(H(\tilde\gamma(\cdot,s),\frac{1}{2}\mc B^t))-E(\tilde\gamma(\cdot,t))+E(H(\tilde\gamma(\cdot,t),\frac{1}{2}\mc B^t))\vert\leq\frac{1}{4}e_{\tilde\gamma,\eps_0/4}(t).
\end{equation}

Since $I^t$ is compact, we can cover $I^t$ be finitely many $I^t$'s, say $I^{t_1},\cdots,I^{t_m}$. By discarding some of them, we can arrange that each $t$ lies in at least one $\overline{I^{t_j}}$ and at most two consecutive $\overline{I^{t_j}}$. In other word, we get a family of intervals $I^{t_j}$'s such that $I^{t_j}$ only intersects $I^{t_{j-1}}$ and $I^{t_{j+1}}$, and $I^{t_{j-1}}$ and $I^{t_{j+1}}$ does not intersect each other. Now for each $j=1,\cdots,m$, we choose a continuous function $r_j(t):[0,1]\to[0,1]$ so that $r_j(t)=1$ on $\overline{I^{t_j}}$ and $r_j(t)=0$ for $t\not\in 2I^{t_j}\cap (I^{t_{j-1}}\cup I^{t_{j}}\cup I^{t_{j+1}})$.
\vskip 1.5mm

Property (1) follows directly, and property (2) follows from the estimate in the proof.
\end{proof}

Use this construction we can prove Theorem \ref{T:Tightening theorem}. 

\begin{proof}[Proof of Theorem \ref{T:Tightening theorem}]
Let $\mc B_1,\cdots,\mc B_m$ and $r_1,\cdots,r_m$ be given by Lemma \ref{lem3.24}. We will use an $m$-step replacement process to define $\gamma$. \vskip 1.5mm

Let us set $\gamma^0=\tilde\gamma$. Then define $\gamma^j$ by applying harmonic replacement to each slice of $\gamma^{j-1}$, set $\gamma^{j}(\cdot,t)=H(\gamma^{j-1}(\cdot,t),r_k(t)\mc B_k)$. We set $\gamma=\gamma^m$.\vskip 1.5mm

We first claim this is a well-defined process, and $\gamma$ is again in $\Omega_{\tilde\gamma}$. In fact, property (1) in Lemma \ref{lem3.24} implies that each energy minimizing map replaces a map with energy at most $2\eps_0/3<\eps_0$. Moreover, Theorem \ref{T:continuity of harmonic replacement} implies the replacement depends continuously on the boundary values, which are themselves continuous in $t$. It is clear $\gamma$ is homotopic to $\tilde\gamma$ since continuously shrinking the disjoint closed balls on which we make harmonic replacement gives an explicit homotopy. \vskip 1.5mm

Now we show this $\gamma$ satisfies the requirement of Theorem \ref{T:Tightening theorem}. Suppose $t\in[0,1]$ is chosen with $E(\tilde\gamma(\cdot,t))\geq W/2$, then property (2) of Lemma \ref{lem3.24} implies that the harmonic replacement for $\tilde \gamma(\cdot,t)$ on $(r_j(t)/2)\mc B_{j(t)}$ decreases the energy by at least $e_{\tilde\gamma,\eps_0}(t)/8$. Thus from Lemma \ref{lem3.8} we get
\begin{equation}\label{3.28}
    E(\tilde\gamma(\cdot,t))-E(\gamma(\cdot,t))\geq k(\frac{1}{8}e_{\tilde\gamma,\eps_0/8}(t))^2.
\end{equation}

Suppose that $\mc B$ is a finite collection of disjoint generalized closed balls in $D$ so that the energy of $\gamma(\cdot,t)$ on $\mc B$ is at most $\eps_0/12$. We can assume $\gamma^k(\cdot,t)$ has energy at most $\eps_0/8$ on $\mc B$ for every $k$, otherwise we can chose $\Psi$ to be a linear function $\Psi(x)=Cx$ for $C$ large. Now we apply (\ref{E:lem3.8 second formula}) in Lemma \ref{lem3.8} twice with $\mu=1/8$ and then $\mu=1/4$ to get
\begin{equation}\label{3.29}
\begin{split}
       E(\gamma(\cdot,t))-E(H(\gamma(\cdot,t),\frac{1}{8}\mc B))&\leq E(\tilde\gamma(\cdot,t))-E(H(\tilde\gamma(\cdot,t),\frac{1}{2}\mc B))+\frac{2}{k}(E(\tilde\gamma(\cdot,t))-E(\gamma(\cdot,t)))^{1/2}\\
       &\leq e_{\tilde\gamma,\eps_0/8}(t)+\frac{2}{k}(E(\delta\gamma(\cdot,t))-E(\gamma(\cdot,t)))^{1/2}.
\end{split}
\end{equation}
Combining (\ref{3.28}) and (\ref{3.29}) with Theorem \ref{T:convexity2}, we complete the proof by choosing $\Psi(x)=Cx+Cx^{1/2}$ for some fixed large $C$ and $\eps_1=\frac{1}{2}\eps_0$.
\end{proof}

\section{Compactness of maximal slices}
\label{S:Compactness of maximal slices}

This section is devoted to the proof of Theorem \ref{T:Compactness theorem}. In particular, we will prove that any sequence of maps (as slices of approximating sweepouts) whose energy converges to the width will converge to a bubble tree of free boundary harmonic disks and harmonic spheres. Similar bubble tree convergence was first studied by Fraser \cite{Fraser00} for $\alpha$-harmonic disks with free boundary, (where the notion of $\alpha$-maps was introduced by Sacks-Uhlenbeck \cite{Sacks-Uhlenbeck81}). In our proof, we have to adopt the schemes in both \cite[Appendix B]{CM08} and \cite{Fraser00}. As an essential ingredient, we generalize the notion of almost harmonic maps, their asymptotic analysis and compactness in \cite[Appendix B]{CM08} to our free boundary setting.

\subsection{Compactness of free boundary almost harmonic maps}

We first introduce our notion of almost harmonic maps with free boundary. Note that generalized balls were defined in Definition \ref{D:generalized balls}.
\begin{defi}
\label{D:almost harmonic with free boundary}
We say a sequence of maps $u^j: (D, \partial D)\to (\mc N, \Gamma)$ is {\em $\eps_1$-almost harmonic in the free boundary sense} if 
\begin{itemize}
\item[$(B_0)$:] for any generalized ball $B\subset D$ with $\int_B\vert\nabla u^j\vert^2<\eps_1$, there is a free boundary harmonic replacement $v:\frac{1}{8}B\to \mc N$ of $u^j$ with free boundary along $v(\partial \frac{1}{8}B \cap \partial D)\subset \Gamma$ (which may be empty), 
which satisfies the following bound:
\[ \int_{\frac{1}{8}B}\vert\nabla u^j-\nabla v\vert^2\leq\frac{1}{j}. \]
\end{itemize}
\end{defi}

Next we have the following preliminary compactness result for a sequence of almost harmonic maps with free boundary and finite energy. In particular, any such sequence converges to a harmonic map with free boundary in the $W^{1, 2}$-norm locally away from only finitely many points in $D$. Let $\eps_{SU}$ and $\eps_F$ be the small thresholds of the $\eps$-regularity results for harmonic maps $u:D\to \mc N$ or free boundary harmonic maps $u: (D, \partial D)\to (\mc N, \Gamma)$ given in \cite[Main Estimate 3.2]{Sacks-Uhlenbeck81} and \cite[Proposition 1.7]{Fraser00} respectively.
\begin{thm}
\label{T:Compactness of free boundary almost harmonic maps}
Let $\eps_1>0$ be less than $\eps_{SU}$ and $\eps_{F}$. Assume $u^j: (D, \partial D)\to (\mc N, \Gamma)$ is a sequence of $W^{1,2}$-maps satisfying property $(B_0)$ in Definition \ref{D:almost harmonic with free boundary} and $E(u^j)\leq E_0<\infty$. Then there exist a finite collection of points $\{x_1,\cdots, x_k\}\subset \overline{D}$, and a subsequence of maps (still denoted by $u^j$), and a harmonic map $u: (D, \partial D) \to (\mc N, \Gamma)$ with free boundary, so that $u^j\to u$ weakly in $W^{1,2}(D)$, and moreover for any compact subset $K\subset \overline D\setminus\{x_1,\cdots,x_k\}$, $u^j\to u$ strongly in $W^{1,2}(K)$. Furthermore, the measures $\vert\nabla u^j\vert^2 dx$ converges to a measure $\nu$ on $\overline{D}$ with $\nu(\{x_i\})\geq \eps_1$ for all $1\leq i\leq k$ and $\nu(\overline{D})\leq E_0$.
\end{thm}

\begin{proof}
After passing to a subsequence we can assume that the sequence of $u^j$'s converges weakly in $W^{1,2}(D)$ to a $W^{1,2}$-map $u:D\to \mc N$, and the measures $\vert\nabla u^j\vert^2 dx$ converges weakly to a limiting measure $\nu$ on $\overline{D}$ with $\nu(\overline{D})\leq E_0$.
\vskip 1.5mm

So there are at most $E_0/\eps_1$ points $x_1,\cdots,x_k \in \overline{D}$, with $\lim_{r\to 0}\nu(B_r(x_j))\geq\eps_1$.\vskip 1.5mm

Next we show that away from these points the convergence is strong in the $W^{1,2}$-norm and $u$ is harmonic with free boundary in $\Gamma$. Given $x\in \overline{D}\backslash \{x_1,\cdots,x_k\}$, then by definition there exist a generalized ball $B_x$ and an integer $J_x$ so that $\int_{B_{x}}\vert\nabla u^j\vert^2<\eps_1$ for $j>J_x$. If $x$ is an interior point we can choose $B_x$ to be a classical ball which is contained in $D$. By condition $(B_0)$ we get a free boundary harmonic replacement $v_x^j:\frac{1}{8}B_{x}\to \mc N$, such that
\[ \int_{\frac{1}{8}B_{x}}\vert\nabla u^j-\nabla v_x^j\vert\leq\frac{1}{j}. \]

Inside the ball $\frac{1}{8}B_{x}(x)$ the energy $E(v_x^j)\leq \eps_1$, so by the $\eps$-regularity estimate by Sacks-Uhlenbeck \cite[Main Estimate 3.2]{Sacks-Uhlenbeck81} and Fraser \cite[Proposition 1.7]{Fraser00}, we get uniform $C^{1,\alpha}$-bound for $v_x^j$ in $\frac{1}{9}B_{x}$ . Hence a subsequence $v_x^j$ converges strongly in $W^{1,2}(\frac{1}{9}B_{x})$ to a harmonic map $v_x: \frac{1}{9}B_{x}\to \mc N$ with free boundary along $v_x(\partial \frac{1}{9}B_{x}\cap \partial D)\subset \Gamma$. By triangle inequality we get
\[ \int_{\frac{1}{9}B_{x}}\vert\nabla u^j-\nabla v_x\vert^2\leq 2\int_{\frac{1}{9}B_{x}}\vert\nabla u^j-\nabla v_x^j\vert^2+2\int_{\frac{1}{9}B_{x}}\vert\nabla v_x^j-\nabla v_x\vert^2\to 0, \]
as $j\to\infty$.
\vskip 1.5mm

We can also derive the $L^2$-convergence of $u^j$ to $v_x$ on $\frac{1}{9}B_{x}$ by the following inequality
\[\int_{\frac{1}{9}B_{x}}\vert u^j-v_x\vert^2\leq 2\int_{\frac{1}{8}B_{x}}\vert u^j-v_x^j \vert^2+2\int_{\frac{1}{9}B_{x}}\vert v^j_x-v_x\vert^2, \]
and the Poincar\'e inequality (when $B_{x}$ is an interior ball of $D$) or its variant Lemma \ref{L:Poincare inequality with partial zero boundary values} (when $B_{x}$ is a boundary ball).
\vskip 1.5mm

Therefore we proved that the sequence $u^j$ converges to $v_x$ strongly in $W^{1,2}(\frac{1}{9}B_{x})$, hence $u=v_x$ in $\frac{1}{9}B_{x}$. So we conclude that $u$ is a free boundary harmonic map on $D\setminus\{x_1,\cdots,x_k\}$. Furthermore, for $K$ relative compact in $D\setminus\{x_1,\cdots,x_k\}$, the $W^{1, 2}(K)$-convergence of $u^j$ to $u$ follows from a covering argument.\vskip 1.5mm

Finally, since $u$ has finite energy, we can apply the removable singularity theorem by Sacks-Uhlenbeck in \cite[Theorem 3.6]{Sacks-Uhlenbeck81} or Fraser in \cite[Theorem 1.10]{Fraser00} at each $x_i$ for interior points and free boundary points respectively. So $u$ extends to a global harmonic map on the whole $D$ with free boundary along $u(\partial D)\subset\Gamma$.
\end{proof}

\subsection{Harmonic maps on half cylinders}
In this and the following subsections, we will generalize the analysis of harmonic maps defined on cylinders in \cite[Appendix B]{CM08} to harmonic maps with free boundary defined on half cylinders. The analysis of harmonic maps and almost harmonic maps on cylinders in \cite{CM08} is essential to the later proof of energy identity. More precisely, in the blow up process (see Section \ref{SS:blowup analysis}), the energy would lose (then the energy identity fails) only when some energy escapes from the ``necks" (modeled by cylinders or half cylinders), and this is the case we want to rule out in our scenario. Hence we need to study carefully the maps defined on the necks, i.e. the cylinders or half cylinders.\vskip 1.5mm

In the free boundary setting, not only spherical bubbles but also disk bubbles may appear during the blow up process. So we need to analyze the  ``necks" between the spheres and disks, i.e. cylinders and half cylinders.\vskip 1.5mm

Let us set up a few notations. We will use $\mc C_{a,b}$ to denote the flat half cylinder $[a,b]\times [0,\pi]$, where $[0, \pi]$ can be viewed as half of a circle ${\bf S}^1$. We will use $(t, \theta)$ as parameters on $\mc C_{a,b}$. Note it is conformally equivalent to the half annulus in plane: $[e^{a},e^{b}]\times [0,\pi]$ under polar coordinates, see Figure \ref{pic2}. We will also abuse the notation $\partial^C \mc C$ to denote the part of the boundary $[a,b]\times \{0,\pi\}$, and later this would be the free boundary part of our maps and we only care about the boundary behavior of the maps along this part. Moreover, when we say a map $u: \mc C\to \mc N$ is a harmonic maps with free boundary on $\Gamma$, we always assume $u\vert_{\partial^C \mc C}$ is the free boundary with $u(\partial^C \mc C)\subset \Gamma$. 
\begin{figure}[ht]
\centering
\ifpdf
  \setlength{\unitlength}{1bp}%
  \begin{picture}(401.13, 86.30)(0,0)
  \put(0,0){\includegraphics{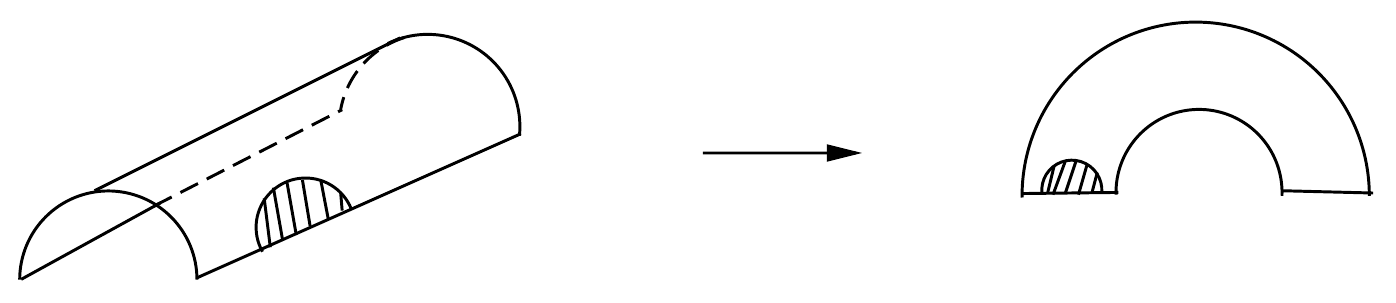}}
  \put(200.48,53.14){\fontsize{9.39}{11.27}\selectfont Conformal}
  \end{picture}%
\else
  \setlength{\unitlength}{1bp}%
  \begin{picture}(401.13, 86.30)(0,0)
  \put(0,0){\includegraphics{2}}
  \put(200.48,53.14){\fontsize{9.39}{11.27}\selectfont Conformal}
  \end{picture}%
\fi
\caption{\label{pic2}%
 }
\end{figure}

\begin{thm}
\label{T:theta energy estimates for free boundary harmonic maps}
Given $\delta>0$, there exist $\eps_2>0$ and $l>1$ depending on $\delta$ so that if $u$ is a non-constant $C^3$-harmonic map from the flat half cylinder $\mc C_{-3l, 3l}$ to $\mc N$ with free boundary along $\Gamma$, and the energy $E(u)\leq\eps_2$, then
\begin{equation}
\label{E:theta energy estimates for free boundary harmonic maps}
\int_{\mc C_{-l,l}}\vert u_\theta\vert^2<\delta\int_{\mc C_{-2l,2l}}\vert\nabla u\vert^2.
\end{equation}
\end{thm}

Roughly speaking, this theorem implies that for non-constant harmonic maps with small energy, the $\theta$-energy (on a sub-half cylinder) is far smaller than the total energy. Then by the Cauchy-Schwartz inequality this implies that the area of the image of $u$ is strictly less than the energy of $u$. \vskip 1.5mm

In order to prove the theorem, we follow the idea of Colding-Minicozzi to prove a differential inequality for various energies of a free boundary harmonic map. In particular we need some free boundary versions of lemmas in \cite[Appendix B]{CM08}.

\begin{lemma}
\label{L:ODinequality for angular energy}
Given a $C^3$-free boundary harmonic map $u$ from $\mc C_{-3l,3l}$ to $\mc N\subset\mb R^N$, with $E(u)\leq\eps_2$ for some small $\eps_2$, then
\begin{equation}
\partial_t^2\int_t\vert u_\theta\vert^2\geq \frac{1}{C}\int_t\vert u_\theta\vert^2 - C\int_t\vert\nabla u\vert^4 - C \sqrt{\eps_2}\frac{1}{\lambda}\int_{t-\lambda}^{t+\lambda}\int_s |\nabla u|^2 ds
\end{equation}
holds for all $t\in[-2l,2l]$ and $\lambda\in(0,1)$. Here $\int_t (\cdot)$ means the integral over the $t$-slice $\{(t, \theta)\in \mc C_{-3l,3l}: \theta\in [0, \pi]\}$, and $C$ is some universal constant depending only on $\mc N$ and $\Gamma$. 
\end{lemma}

\begin{proof}
We can first assume $\eps_2\leq \eps_0$ (as in Theorem \ref{e-reg}). Differentiating $\int_t\vert u_\theta\vert^2$ twice gives
\begin{align}
\frac{1}{2}\partial_t^2\int_t\vert u_\theta\vert^2&=\int_t\vert u_{t\theta}\vert^2+\int_t\langle u_\theta,u_{tt\theta}\rangle\notag\\
&=\int_t\vert u_{t\theta}\vert^2-\int_t\langle u_{\theta\theta},u_{tt}\rangle+\langle u_\theta, u_{tt} \rangle \vert_0^\p\notag\\
&=\int_t\vert u_{t\theta}\vert^2-\int_t\langle u_{\theta\theta},(\Delta u-u_{\theta\theta})\rangle + \langle u_\theta, A^\Gamma(u_t, u_t)\rangle\vert_0^\pi\\
&\geq \int_t\vert u_{t\theta}\vert^2+\int_t\vert u_{\theta\theta}\vert^2-\sup_{\mc N}\vert A^{\mc N}\vert\int_t\vert u_{\theta\theta}\vert\vert\nabla u\vert^2 - 2\sup_\Gamma |A^\Gamma|\, C\sqrt{\eps_2} \frac{1}{\lambda}\int_{t-\lambda}^{t+\lambda}\int_s |\nabla u|^2 ds.\notag
\end{align}
Here $A^\Gamma$ and $A^{\mc N}$ are the second fundamental forms of $\Gamma\hookrightarrow\mc N$ and $\mc N\hookrightarrow \mb R^N$. In the second equality we used integration by parts; in the third inequality since $u$ is a free boundary map, on the boundary (where $\theta=0$ or $\theta=\pi$) $u_{\theta}$ is a tangent vector of $\mc N$ and also perpendicular to $\Gamma$, and $u_{tt}=\nabla^{\mc N}_{u_t}u_t$ when projected to the tangent space $T\mc N$, where $u_t$ is tangent to $\Gamma$; in the last inequality we used $\vert\Delta u\vert\leq \sup_{\mc N}\vert A^{\mc N}\vert \vert\nabla u\vert^2$ by the harmonic map equation (\ref{E:harmonic map equation}), and also the gradient estimate (Theorem \ref{e-reg}) applied to half disks of radius $\lambda$ centered at $(t, 0)$ and $(t, \pi)$. By the Cauchy-Schwartz inequality we get
\[\frac{1}{2}\partial_t^2\int_t\vert u_\theta\vert^2\geq \int_t\vert u_{t\theta}\vert^2+\frac{3}{4}\int_t\vert u_{\theta\theta}\vert^2- C\int_t\vert\nabla u\vert^4 - C \sqrt{\eps_2} \frac{1}{\lambda}\int_{t-\lambda}^{t+\lambda}\int_s |\nabla u|^2 ds, \]
for some universal constant $C$ depending only on $\mc N$ and $\Gamma$.\vskip 1.5mm

We will prove the following statement: for $u$ satisfying the assumption of the theorem,\begin{equation}
\label{E:estimates of angular energy by integral of twice derivative}
\int_t\vert u_{\theta}\vert^2\leq C \int_t\vert u_{\theta\theta}\vert^2, \text{ for } t\in[-2l,2l],
\end{equation}
where $C>0$ is a constant depending only on $\mc N$ and $\Gamma$. Note that once we have this inequality we get the desired inequality in the lemma.\vskip 1.5mm

To prove this inequality, we reflect the map $u(t, \cdot): [0, \pi]\to \mb R^N$ across $\Gamma$ to obtain a map defined on the circle $\hat u(t, \cdot): \mb S^1\to \mb R^N$. In particular, denote $P_\Gamma$ as the nearest point projection map from a tubular neighborhood of $\Gamma$ in $\mb R^N$ to $\Gamma$. When the tubular neighborhood is chosen small enough, we can assume that
\[ |DP_\Gamma|\leq 1, \text{ and } |D^2P_\Gamma|\leq C, \]
for some universal constant $C>0$. By the gradient estimate Theorem \ref{e-reg}, we can assume the image of $u(\mc C_{-2l ,2l})$ lie in this tubular neighborhood when $\eps_2$ is chosen small enough. Now define the map $\hat u(t, \cdot): \mb S^1\to \mb R^N$ by
\[ \hat u(t, \theta)=\begin{cases} u(t, \theta), & \text{ when } \theta\in [0, \pi]\\ 
                                                   2P_\Gamma(u(t, -\theta))-u(t, -\theta), & \text{ when } \theta\in [-\pi, 0].
                                                   \end{cases}\]
By the free boundary assumption, we know that $\hat u$ is $C^{1, 1}$ on $\mb S^1$. As $\int_{\mb S^1} \hat u_\theta =0$, by the Wirtinger inequality, we get $\int_{\mb S^1} \vert \hat u_\theta\vert^2\leq \int_{\mb S^1} \vert \hat u_{\theta\theta}\vert^2$, and hence we can deduce that
\begin{equation}
\begin{split}
\int_0^\pi \vert u_\theta\vert^2  & \leq \int_0^\pi \vert u_{\theta\theta}\vert^2 + \int_0^\pi \vert (2P_\Gamma(u(t, \theta))-u(t, \theta))_{\theta\theta}\vert^2\\
                                                 &\leq \int_0^\pi \vert u_{\theta\theta}\vert^2 + \int_0^\pi \vert 2DP_\Gamma(u_{\theta\theta})-u_{\theta\theta}+2D^2P_\Gamma(u_\theta, u_\theta) \vert^2\\
                                                 &\leq C \int_0^\pi \vert u_{\theta\theta}\vert^2+ C \int_0^\pi |u_{\theta}|^4.
\end{split}
\end{equation}
Note that by the gradient estimates Theorem \ref{e-reg}, $C \int_0^\pi |u_{\theta}|^4\leq C^2 \eps_2\int_0^\pi |u_\theta|^2$. So the desired estimates (\ref{E:estimates of angular energy by integral of twice derivative}) follows by taking $\eps_2$ small enough, so that $C^2\eps_2<\frac{1}{2}$.
\end{proof}

Next lemma is an ODE comparison lemma.

\begin{lemma}\label{L:ODE lemma}
Suppose $f$ be a nonnegative $C^2$ function on $[-2l,2l]\subset\mb R$ satisfying
\begin{equation}\label{E:ODE in ODE lemma}
f''\geq \frac{1}{4C}f-a,
\end{equation}
for some constants $C, a>0$. If $\max_{[-l,l]}f\geq 8 C a$, then
\begin{equation}
\int_{-2l}^{2l}f\geq 4\sqrt{2C}a\sinh(2\sqrt{C}l).
\end{equation} 
\end{lemma}

\begin{proof}
Let $\tilde f(t)=f(t/\sqrt{4C})$, then we get a differential inequality of $\tilde f$:
\[\tilde f''\geq \tilde f-4Ca,\]
where $\tilde f$ is defined on $[-4\sqrt{C}l,4\sqrt{C}l]$, and $\max_{[-2\sqrt{C}l,2\sqrt{C}l]}\tilde f\geq 8Ca$. Then applying \cite[Lemma B.4]{CM08} gives
\[\int_{-4\sqrt{C}l}^{4\sqrt{C}l}\tilde{f}\geq 8\sqrt{2}Ca\sinh (2\sqrt{C}l),\]
which implies
\[\int_{-2l}^{2l}f\geq 4\sqrt{2C}a\sinh(2\sqrt{C}l).\qedhere\]
\end{proof}

Now we can prove the main theorem of this subsection.

\begin{proof}[Proof of Theorem \ref{T:theta energy estimates for free boundary harmonic maps}]
First we analyze $\int_t(\vert u_t\vert^2-\vert u_\theta\vert^2)$. Differentiating it and applying integration by parts gives
\begin{equation}
\begin{split}
\frac{1}{2}\partial_t \int_t(\vert u_t\vert^2-\vert u_\theta\vert^2)&=\int_t(\langle u_t, u_{tt}\rangle-\langle u_\theta,u_{\theta t}\rangle)\\
&=\int_t\langle u_t,u_{tt}+u_{\theta\theta}\rangle \,=\, 0.
\end{split}
\end{equation} 
Here again we use the fact that $u_\theta$ and $u_t$ are perpendicular to each other on the free boundary, and $u_{tt}+u_{\theta\theta}=\Delta u$ which is normal to $\mc N$ and hence perpendicular to $u_t$. Thus $\int_t(\vert u_t\vert^2-\vert u_\theta\vert^2)$ is a constant, and
\begin{equation}
\int_t(\vert u_t\vert^2-\vert u_\theta\vert^2)=\frac{1}{4l}\int_{\mc C_{-2l,2l}}(\vert u_t\vert^2-\vert u_{\theta}\vert^2)\leq \frac{1}{4l}\int_{\mc C_{-2l,2l}}\vert\nabla u\vert^2.
\end{equation}

Moreover, we get

\begin{equation}
\int_t\vert\nabla u\vert^2=2\int_t \vert u_\theta\vert^2+\int_t(\vert u_t\vert^2-\vert u_\theta\vert^2)\leq 2\int_t\vert u_\theta\vert^2+ \frac{1}{4l}\int_{\mc C_{-2l,2l}}\vert\nabla u\vert^2.
\end{equation}

Let us choose $\eps_2$ smaller than $\eps_0$ for $\eps_0$ in \cite[Lemma 3.4]{Sacks-Uhlenbeck81} and Theorem \ref{e-reg}. Then the interior gradient estimates for harmonic maps (see \cite[Lemma 3.4]{Sacks-Uhlenbeck81}) and for free boundary harmonic maps (Theorem \ref{e-reg}) imply that
\begin{equation}
\sup_{\mc C_{-2l, 2l}}\vert\nabla u\vert^2\leq C\eps_2.
\end{equation}

Let $f(t)=\int_t\vert u_\theta\vert^2$, then by Lemma \ref{L:ODinequality for angular energy} we get
\begin{align}
f''(t) & \geq \frac{1}{C}f(t) - C \eps_2 \int_t\vert\nabla u\vert^2 - C \sqrt{\eps_2}\frac{1}{\lambda} \int_{t-\lambda}^{t+\lambda}\int_s |\nabla u|^2 ds\notag\\
       & \geq \frac{1}{C}f(t)-2C\eps_2f(t)-2C\sqrt{\eps_2}\frac{1}{\lambda}\int_{t-\lambda}^{t+\lambda}f(t)-\frac{C\eps_2}{4l}\int_{\mc C_{-2l,2l}}\vert\nabla u\vert^2-\frac{2C\sqrt{\eps_2}}{4l}\int_{\mc C_{-2l,2l}}\vert\nabla u\vert^2\notag\\
       & \geq \frac{1}{2C}(f(t)-\frac{1}{4\lambda}\int_{t-\lambda}^{t+\lambda}f)-\frac{C\sqrt{\eps_2}}{l}\int_{\mc C_{-2l,2l}}\vert\nabla u\vert^2,
\end{align}
where $C>0$ depends only on $\mc N$ and $\Gamma$, and we can further assume that $\eps_2$ is small enough so that $C^2\eps_2\leq 1/4$ and $2C^2\eps_2^{1/4}\leq 1$.

Let $\lambda\to 0$, by continuity of $f$ we get the differential inequality
\begin{equation}
f''(t)\geq \frac{1}{4C}f(t) -a,
\end{equation}
where $a= \frac{C\sqrt{\eps_2}}{l}\int_{\mc C_{-2l,2l}}\vert\nabla u\vert^2$. Then we apply Lemma \ref{L:ODE lemma} to get either
\begin{equation}
\max_{[-l,l]}f<8\frac{C^2\sqrt{\eps_2}}{l}\int_{\mc C_{-2l,2l}}\vert\nabla u\vert^2,
\end{equation}
or
\begin{equation}
\int_{\mc C_{-2l,2l}}\vert u_\theta\vert^2=\int_{-2l}^{2l}f(t)dt\geq  4\sqrt{2C}C\sqrt{\eps_2}\frac{\sinh(2\sqrt{C}l)}{l}\int_{\mc C_{-2l,2l}}|\nabla u|^2.
\end{equation}

If we choose $l$ large enough then the second inequality can not hold. Then we get
\begin{equation}
\int_{\mc C_{-l,l}}\vert u_\theta\vert^2\leq 2l\max_{[-l,l]} f<8C^2\sqrt{\eps_2}\int_{\mc C_{-2l,2l}}\vert\nabla u\vert^2.
\end{equation}

Then the inequality (\ref{E:theta energy estimates for free boundary harmonic maps}) holds if $8C^2\sqrt{\eps_2}<\delta$. So we can choose $\eps_2$ small and then choose $l$ large to get the desired inequality.
\end{proof}

\subsection{Almost Harmonic Maps on Half Cylinders}

The main results in this part generalize the results in the previous subsection to almost harmonic maps on half cylinders. \vskip 1.5mm

Let us first fix some notation. Given a half cylinder $\mc C_{r_1,r_2}$, we will view it as its conformally equivalent half annulus $D^+_{e^{r_2}}\setminus D^+_{e^{r_1}}\subset D^+$. A generalized ball $B\subset D^+_{e^{r_2}}\setminus D^+_{e^{r_1}}$ is either a ball in the interior of the annulus $D^+_{e^{r_2}}\setminus D^+_{e^{r_1}}$, or is a half ball centered along the chord boundary $\partial^C_{e^{r_2}}\setminus \partial^C_{e^{r_1}}$. When $B$ is a half ball, we denote $\partial^C B=\partial B\cap \partial^C$. Note that this definition is the same as that in Definition \ref{D:generalized balls}.

\begin{defi} 
Given $\mu>0$ and a half cylinder $\mc C_{r_1,r_2}$, we will say that a $W^{1,2}$-map $u: \mc C_{r_1,r_2} \to \mc N$ with $u(\partial^C \mc C)\subset \Gamma$ is {\em $\mu$-almost harmonic with free boundary} if for any finite collection of disjoint generalized balls $\mc B$ in the conformally equivalent half annulus $D^+_{e^{r_2}}\setminus D^+_{e^{r_1}}$, there is a free boundary harmonic replacement $v:\cup_{\mc B}\frac{1}{8}B\to \mc N$ with free boundary along $u\left(\cup_\mc B\frac{1}{8}\partial^C B\right)\subset \Gamma$, such that
\[\int_{\cup_{\mc B}\frac{1}{8}B}\vert\nabla u-\nabla v\vert^2\leq\frac{\mu}{2}\int_{\mc C_{r_1,r_2}}\vert\nabla u\vert^2.\]
\end{defi}

The first lemma of this subsection shows that for almost harmonic maps with free boundary the estimate in previous subsection still holds.

\begin{lemma}
\label{L:theta energy estimates for free boundary almost harmonic maps}
Given $\delta>0$ there exists $\mu>0$ depending on $\delta$, $\mc N$ and $\Gamma$ so that if $u: \mc C_{-3l,3l} \to \mc N$ is a $\mu$-almost harmonic map with free boundary, and $E(u)\leq \eps_2$, then
\begin{equation}
\int_{\mc C_{-l,l}}\vert u_\theta\vert^2\leq\delta\int_{\mc C_{-3l,3l}}\vert\nabla u\vert^2.
\end{equation}
\end{lemma}

\begin{proof}
We will argue by contradiction. Suppose the lemma does not hold, then there exists a sequence $u^j$ of $\frac{1}{j}$-almost harmonic maps from $\mc C_{-3l,3l}$ to $\mc N$ with free boundary $u_j(\partial^C \mc C_{-3l, 3l})\subset \Gamma$, $E(u^j)\leq\eps_2$, and
\begin{equation}
\label{E:contradiction sequence for almost harmonic maps}
\int_{\mc C_{-l,l}}\vert u^j_\theta\vert^2>\delta\int_{\mc C_{-3l,3l}}\vert\nabla u^j\vert^2.
\end{equation}

Now we have two cases depending on whether the energy of limit is zero.\vskip 0.5em

{\bf Case 1:} Suppose $\limsup_{j\to\infty}E(u^j)>0$, then up to a subsequence $\int_{\mc C_{-l,l}}\vert u^j_\theta\vert^2$ is uniformly bounded from below by (\ref{E:contradiction sequence for almost harmonic maps}). We will apply the compactness result (Theorem \ref{T:Compactness of free boundary almost harmonic maps}) to this sequence. In particular, we can use the same argument as Theorem \ref{T:Compactness of free boundary almost harmonic maps} to find a subsequence that converges weakly to a free boundary harmonic map $u: (\mc C_{-3l,3l}, \partial^C \mc C_{-3l, 3l}) \to (\mc N, \Gamma)$, and strongly in $W^{1,2}$ on any compact subset of $\mc C_{-3l,3l}$. Note that since $E(u^j)\leq\eps_2$, there will be no energy concentration points. By the uniform lower bound of $\int_{\mc C_{-l,l}}\vert u^j_\theta\vert^2$ and $W^{1, 2}$-convergence on $\mc C_{-l,l}$, $u$ can not be a constant map. Finally by the lower semi-continuity of energy along $W^{1, 2}$-weak convergence,
\[\int_{\mc C_{-l,l}}\vert u_\theta\vert^2 \geq \delta\int_{\mc C_{-3l,3l}}\vert\nabla u\vert^2,\]
which contradicts Theorem \ref{T:theta energy estimates for free boundary harmonic maps}.\vskip 0.5em

{\bf Case 2:} Suppose $\lim_{j\to\infty}E(u^j)=0$. We will use a blow-up argument. Let 
\[v^j=\frac{u^j-u^j(0)}{E(u^j)^{1/2}}.\]
This is a sequence of maps from $\mc C_{-3l, 3l}$ to $\mc N_j=(\mc N-u^j(0))/(E(u^j)^{1/2})$. Here $0=(0,0)$ is a boundary point on $\mc C_{-3l, 3l}$, so $u^j(0)\in\Gamma$; hence we can always see the free boundary $\Gamma_j=(\Gamma-u^j(0))/(E(u^j)^{1/2})$ in the blowup process.

Note that $E(v^j)=1$ and by (\ref{E:contradiction sequence for almost harmonic maps}) we have,
\[\int_{\mc C_{-l,l}}\vert v_\theta^j\vert^2>\delta>0.\]

Furthermore, the sequence of $v^j$'s are still $1/j$-almost harmonic because this property is invariant under dilation. So we can argue as before to get a subsequence that converges in $W^{1,2}$ on compact subset of $\mc C_{-3l,3l}$ to a free boundary harmonic map $v: \mc C_{-3l,3l}\to\mb R^n\subset \mb R^N$ with free boundary $\Gamma=\mb R^k\subset\mb R^N$. 
As before, we get
\[\int_{\mc C_{-l,l}}\vert v_\theta\vert^2\geq \delta,\]
which is again a contradiction to Theorem \ref{T:theta energy estimates for free boundary harmonic maps}; (note that for free boundary harmonic maps into $(\mb R^n, \mb R^k)$ we do not need the assumption $E(v)\leq \eps_2$; see also \cite[Remark B.7]{CM08}).
\end{proof}

With this lemma we can prove that the $\theta$-energy of a free boundary almost harmonic maps on a long half cylinder would be far less than the total energy.

\begin{thm}\label{T: Theta energy small for almost harmonic maps}
Given $\delta>0$ there exists $\nu>0$ depending on $\delta$, $\mc N$ and $\Gamma$, so that if $m$ is any positive integer and $u$ is $\nu$-almost harmonic from $\mc C_{-(m+3)l,3l}$ to $\mc N$ with free boundary along $u(\partial^C \mc C_{-(m+3)l,3l})\subset \Gamma$, and $E(u)\leq\eps_2$, then
\begin{equation}
\int_{\mc C_{-ml,0}}\vert u_\theta\vert^2\leq 7\delta\int_{\mc C_{-(m+3)l,3l}}\vert\nabla u\vert^2.
\end{equation}
\end{thm}

\begin{proof}
The proof follows by covering $\mc C_{-(m+3)l,3l}$ by sub-half cylinders of length $6l$ together with Lemma \ref{L:theta energy estimates for free boundary almost harmonic maps}. We refer the details to the proof of \cite[proposition B.19]{CM08}. 
\end{proof}

The following simple lemma will be useful in the next subsection.
\begin{lemma}
\label{L:comparison between area and energy on long cylinders}
Suppose $u: \mc C_{-(m+3)l,3l}\to \mc N$ is a map satisfying 
\[\int_{\mc C_{-ml,0}}\vert u_\theta\vert^2\leq\frac{1}{9} E(u),\]
and
\[\int_{\mc C_{-(m+3)l,-ml}\cup \mc C_{0,3l}}\vert\nabla u\vert^2\leq \frac{1}{9}E(u).\]
Then 
\[ \area(u)\leq \frac{8}{9}E(u).\]
\end{lemma} 

\begin{proof}
First note that
\[\area(u)=\int_{C_{-(m+3)l, 3l}}(\vert u_\theta\vert^2\vert u_t\vert^2-\langle u_\theta,u_t\rangle^2)^{1/2}\leq\int_{C_{-(m+3)l, 3l}}\vert u_\theta\vert\vert u_t\vert.\]

By $ab\leq \frac{4}{3}a^2+\frac{3}{16}b^2$, we get
\begin{equation}
\begin{split}
\int_{C_{-ml,0}}\vert u_\theta\vert\vert u_t\vert&\leq \frac{4}{3}\int_{C_{-ml,0}}\vert u_\theta\vert^2+\frac{3}{16}\int_{C_{-ml,0}}\vert u_t\vert^2\\
&\leq \frac{1}{3}\int_{C_{-ml,0}}\vert u_\theta\vert^2+\frac{3}{16}\int_{C_{-ml,0}}\vert u_t\vert^2+\frac{1}{9}E(u)\\
& \leq \frac{7}{9}E(u).
\end{split}
\end{equation}

Combining with the second assumption, 
we have 
\begin{equation}
\area(u)\leq \frac{7}{9}E(u)+\frac{1}{9}E(u)= E(u)-\frac{1}{9}E(u)\leq \frac{8}{9}E(u).\qedhere
\end{equation}
\end{proof}

\subsection{Proof of Theorem \ref{T:Compactness theorem}}
\label{SS:blowup analysis}

Now we are ready to prove the Theorem \ref{T:Compactness theorem}. We will follow the same scheme as in \cite[Section B.6]{CM08}. One key point in our setting is that we may get two different kinds of bubbles. We may get spherical bubbles as in \cite[Section B.6]{CM08}, as well as free boundary disk bubbles as \cite{Fraser00}. Those new techniques developed in previous sections will be essentially used to study these bubbles.\vskip 1.5mm

A boundary ball $B$ of $D$ is always the intersection of a classical ball $B_r(x)$ of $\mb R^2$ with $D$, i.e. $B=B_r(x)\cap D$ for some $x\in \partial D$ and $r>0$; in the following proof we say $r$ is the radius of $B$ and sometime abuse the notation to write $B=B_r(x)$.\vskip 1.5mm

We use $\tilde\Pi$ to denote a fixed conformal map that maps the upper half plane $\overline{\mb H^2}$ to the unit disk $\overline{D}$ which maps $(0, 1, \infty)$ to three given distinct points on $\partial D$. We denote $p^+$ as the image of $\infty$ and $D^-$ as the image of $D^+_1\subset\overline{\mb H^2}$ under the map $\Pi$.  For a given boundary ball $B_r(x)$, we define the \emph{Conformal Dilation of $B_r(x)$} to be the map $\Psi_{r,x}: \overline D\to \overline D$ so that $\Psi_{r,x}=\tilde\Pi\circ \Phi_{r,x}\circ \tilde\Pi^{-1}$, where $\Phi_{r,x}$ is the composition of a dilation of $\overline{\mb H^2}$ by the factor $1/r$ and a translation of $\overline{\mb H^2}$ by $-\tilde\Pi^{-1}(x)$; (note $\tilde\Pi^{-1}(x)$ is a boundary point of $\overline{\mb H^2}$).

\begin{proof}[Proof of Theorem \ref{T:Compactness theorem}]
We divide the whole proof into two parts. The first part is about the bubbling compactness, and the second part is about the energy identity. Note that bubbling convergence with energy identity implies varifold convergence by \cite[Proposition A.3]{CM08}; (although the domain of the maps is the sphere in \cite[Proposition A.3]{CM08}, the proof works in our case for disk domain with no change).  \vskip 0.5em

{\bf Bubbling convergence:} 
Let $u^j$ be a sequence as in the theorem, then property ($\dagger$) implies property ($B_0$) in Definition \ref{D:almost harmonic with free boundary}. By the compactness Theorem \ref{T:Compactness of free boundary almost harmonic maps}, we can find a free boundary harmonic map $v_0: (D, \partial D)\to (\mc N, \Gamma)$ (which maybe trivial), and a finite collection of singular points $\mc S_0\subset \overline D$, such that a subsequence (still denoted as $u^j$) converges to $v_0$ weakly in $W^{1,2}(D)$ and strongly in $W^{1,2}(K)$ for any compact subset $K\subset \overline D\setminus \mc S_0$. The measures $\vert \nabla u^j\vert^2dx$ converges to a measure $\nu_0$ with $\nu_0(\overline D)\leq E_0$ and at each singular point $x\in \mc S_0$, $\nu_0(\{x\})\geq\eps_1$.\vskip 1.5mm

Next we want to renormalize the maps near the singular points. Let us start with boundary points. Suppose $x\in\mc S_0$ lies on $\partial D$. Let $\eps_3>0$ be smaller than $\eps_1/2$ and $\eps_2$. Fix a radius $\rho>0$ so that $x$ is the only singular point in the boundary ball $B_{2\rho}(x)$ and $\int_{B_\rho(x)}\vert\nabla v_0\vert^2\leq\eps_3$. For each $j$, we choose $r_j>0$ to be the smallest radius so that
\[\inf_{y\in B_{\rho-r_j}(x)\cap \partial D}\int_{B_\rho(x)\backslash B_{r_j}(y)}\vert\nabla u^j\vert^2=\eps_3,\]
and choose a point $y_j\in\partial D$ such that $B_{r_j}(y_j)\subset B_\rho(x)$ with $\int_{B_\rho(x)\backslash B_{r_j}(y_j)}\vert\nabla u^j\vert^2=\eps_3$. Since $u_j$ converges strongly to $v_0$ on any compact subset of $B_\rho(x)\backslash \{x\}$, by the energy bound we get $y_j\to x$ and $r_j\to 0$. \vskip 1.5mm

For each $j$, since the energy functional is invariant under conformal changes, the dilated sequence of maps $\tilde u_1^j=u^j\circ \Psi_{r_j,y_j}$ still satisfies the almost harmonic property ($B_0$) in Definition \ref{D:almost harmonic with free boundary}, and they all have the same energy as $u^j$'s. Using the compactness Theorem \ref{T:Compactness of free boundary almost harmonic maps} again, we get a subsequence (still denoted as $\tilde u^j_1$), and a finite collection of singular points $\mc S_1\subset \overline{D}$, and a free boundary harmonic map $v_1:D\to \mc N$, so that $\tilde u_1^j$ converges to $v_1$ weakly in $W^{1,2}(D)$ and strongly in $W^{1,2}(K)$ for any compact $K\subset \overline{D}\backslash \mc S_{1}$. Moreover, the measures $\vert \nabla \tilde u_1^j\vert^2 dx$'s converge to a measure $\nu_1$ on $\overline{D}$. 

The choice of $B_{r_j}(y_j)$ guarantees that $\nu_1(\overline D\backslash\{ p^+\})\leq \nu_0(\{x\})$ and $\nu_1(D^-)\leq \nu_0(\{x\})-\eps_3$. Next we want to show the following claim: 
\vskip 2mm

\textit{Claim}: the maximal energy concentration at any $y\in\mc S_1\backslash\{p^+\}$ is at most $\nu_0(\{x\})-\eps_3$.

\textit{Proof of the claim}: note any such point $y$ satisfies $\nu_1(\{y\})\geq\eps_1>\eps_3$, hence it can only lies in $D^-$. Then the fact $\nu_1(D^-)\leq \nu_0(\{x\})-\eps_3$ implies that $\nu_1(\{y\})\leq \nu_0(\{x\})-\eps_3$.\vskip 2mm

Now we iterate this blowing up construction at every boundary singular point in $\mc S_0$ and $\mc S_1$, and we will get a subsequent singular sets $\mc S_0,\mc S_1,\mc S_2,\cdots$, and dilated sequences of maps $\{u^j\}, \{\tilde u^j_1\}, \{\tilde u^j_2\}, \cdots$, one more singular set after each blowing up process. From the claim we know that this process must terminate after at most $E_0/\eps_3$ steps, and we have in total $m$-singular sets $\mc S_0,\mc S_1, \cdots, \mc S_m$. Then for each sequence of dilated maps $\tilde u^j_\alpha$, $\alpha\in\{0, \cdots, m\}$, there are no boundary singular points away from $\mc S_\alpha$. Lastly we can apply the blowing up process in \cite[Appendix B.6]{CM08} to each $\tilde u^j_\alpha$ at those interior singular points, and finally get that the sequence $u^j$ converges to a collection of free boundary harmonic disks $v_0, v_1, \cdots, v_m: (D, \partial D)\to (\mc N, \Gamma)$ and harmonic spheres $\tilde v_1, \cdots, \tilde v_k: S^2\to \mc N$. Note that these harmonic spheres arise as blowup limits near interior singular points. \vskip 0.5em

{\bf Energy identity:} In this part we will show the summation of the energy of all $v_i$'s is equal to the limit of the energy of $u^j$'s. i.e. no energy was lost in the bubbling convergence process. \vskip 1.5mm

In order to prove this, we need to re-examine what happens to the energy during the blowing up process. The ``no loss of energy" for blowing up of interior singular points has already been proved by Colding-Minicozzi in \cite[Appendix B.6, Step 4]{CM08}, so we only need to analyze the case for blowing up at boundary singular points. \vskip 1.5mm

At each blowing up step, the energy is taken away from a singular point $x$ and then goes to one of two places:
\begin{itemize}
\item it can show up in the new limiting free boundary harmonic disk of a singular point in $\partial D\backslash{\{p^+\}}$,
\item or it can disappear at $p^+$.
\end{itemize}
In the first case, the energy is accounted in the final summation and no energy is lost. So we only need to rule out the energy loss in the second case. With out loss of generality, we can only prove the ``no loss of energy" for the first blow-up process, i.e. for the convergence of $\{\tilde u^j_1\}$. Note that if there is energy loss, then $\nu_1(\overline D\backslash\{p^+\})<\nu_0(\{x\})$.\vskip 1.5mm

We argue by contradiction. Suppose $\nu_1(\overline D\backslash\{p^+\})\leq \nu_0(\{x\})-\hat\delta$ for some $\hat\delta>0$. Note we must have $\hat\delta\leq\eps_3$. Thus we can choose two sequences of radii $s_j>t_j$ so that $A_j=B_{s_j}(y_j)\setminus B_{t_j}(y_j)$ are half cylinder with
\begin{equation}
s_j\to 0,\, \frac{t_j}{r_j}\to \infty,\, \text{ and } \int_{A_j}\vert\nabla u^j\vert^2\geq\hat\delta>0.
\end{equation}

Actually we may choose $s_j$ close to $\rho$ and $t_j$ close to $s_j/\lambda_j$ and $s_j/t_j>\lambda_j$ for a sequence of $\lambda_j\to\infty$. After a conformal change, $A_j$'s are a sequence of half cylinders with length goes to $\infty$. Moreover, there is quite a bit of freedom in choosing $s_j$ and $t_j$, i.e. we can slightly change $s_j,t_j$ a little bit and the above conditions are still satisfied. So we may also assume $u^j$ has small energy near the two ends of the half cylinder $A_j$ as the second condition in Lemma \ref{L:comparison between area and energy on long cylinders}. \vskip 1.5mm

Theorem \ref{T:theta energy estimates for free boundary harmonic maps} (with $\delta=1/63$) together with $\hat \delta\leq\eps_3\leq\eps_2$ implies the theta energy of $u^j$ on $A_j$ is small, so the first condition of Lemma \ref{L:comparison between area and energy on long cylinders} is satisfied for $j$ sufficiently large. Then by Lemma \ref{L:comparison between area and energy on long cylinders}, we get that the area of the image of $u^j$ on $A_j$ must be strictly less than the energy of $u^j$ on $A_j$ for $j$ large, which is a contradiction to the area assumption (\ref{E:area assumption in compactness theorem}). Thus we complete the proof.
 \end{proof}

\section{Modifications for the proof of Theorem \ref{T:main2} and discussions}
\label{S:Necessary modifications for the proof of fixed boundary min-max}

In this part, we record necessary modifications to adopt the proof of Theorem \ref{T:main1} to Theorem \ref{T:main2}. There are only two places where we have to do some modifications. 
\vskip 0.15in
\noindent{\bf Modification for Theorem \ref{T:area width equals energy width}.}
Again, the first step is to show that for given $\gamma(\cdot,t)\in\Omega$, we can approximate it by some $\tilde\gamma(\cdot,t)\in\Omega$ which lies in $C^0\big([0, 1], C^2(\overline{D}, \mc N)\big)$. So we need to do mollifications on $\gamma$. However, direct mollifications as what we did in Section \ref{S:Conformal parametrization} cannot work here, because the end-point maps $\gamma(\cdot,0)=\bar v_0$ and $\gamma(\cdot,1)=\bar v_1$ may change after mollifications. In order to handle this issue, we first mollify the whole family $\gamma(\cdot,t)$ as in Section \ref{S:Conformal parametrization} to get a continuous family of $C^2$ maps $\bar\gamma(\cdot,t): (D, \partial D)\to (\mc N, \Gamma)$; next by reparametrizing $t\to s(t)=(1-2\mu)t+\mu$, we get a new family $\tilde\gamma(\cdot, s)$ which is defined for $s\in[\mu,1-\mu]$. Moreover, since $s(t)\to t$ as $\mu\to 0$ and $\bar\gamma(\cdot,t)$ is $C^0$ as a function of $t$ to $C^2(\overline{D}, \mc N)$, for any given $\eps>0$, we can choose $\mu$ small such that $\max_{s\in[\mu,1-\mu]}\Vert\tilde\gamma(\cdot,s)-\bar\gamma(\cdot,s)\Vert_{C^2} \leq\eps$. \vskip 1.5mm

Note that $\gamma(\cdot, 0)$ and $\gamma(\cdot, 1)$ are both smooth, so by varying the mollification parameter we can connect $\gamma(\cdot,0)$ and $\gamma(\cdot,1)$ to $\tilde\gamma(\cdot, \mu)$ and $\tilde\gamma(\cdot, 1-\mu)$ respectively, and hence get a continuous family of $C^2$-maps $\tilde\gamma(\cdot,s): (D, \partial D)\to (\mc N, \Gamma)$ for $s\in[0,1]$. The slices of $\tilde\gamma(\cdot,s)$ for $s\in[0,\mu]$ are mollifications of $\gamma(\cdot,0)$ and the slices of $\tilde\gamma(\cdot,s)$ for $s\in[1-\mu,1]$ are mollifications of $\gamma(\cdot,1)$. Thus we get a regularization of $\gamma(t)$ which stay close to $\gamma(t)$ in $C^0(\overline D, \mc N)\cap W^{1, 2}(D, \mc N)$. The conformal reparametrization procedure works in the same way as Section \ref{S:Conformal parametrization}. Note that since $\tilde\gamma(\cdot, 0)=\bar v_0$ and $\tilde\gamma(\cdot, 1)=\bar v_1$ are both conformal harmonic maps, the pull-back metrics of $g$ on $\mc N$ are already conformal to the standard metric $g_0$ on $D$, so the conformal reparametrization maps $h(t)$ satisfy that $h(0)=h(1)=id$. Therefore after conformal reparametrization procedure, the end-point maps are still $\bar v_0$ and $\bar v_1$, so it is a legitimate sweepout homotopic to $\gamma$. We know for this family $\tilde\gamma(\cdot,s)$, the area and the energy are close.

\vskip 0.15in
\noindent{\bf Modification for Theorem \ref{T:Tightening theorem}.}
In the tightening process, we only want to pull tight the slice with energy larger than the energy of $\bar v_0$ and $\bar v_1$. 
Therefore, we need to modify the theorem so that in the statement, $(\ast)$ holds for those $t$ with $E(\tilde\gamma(\cdot,t))\geq W/\lambda$, where $\lambda>1$ is chosen with $W/\lambda>\max(\area(\bar v_0), \area(\bar v_1))$. All the proofs in Section \ref{S:Construction of the tightening process} work in exactly the same way by changing ``$W/2$" to ``$W/\lambda$".

\vskip 0.15in
\noindent{\bf Modification for Theorem \ref{T:Compactness theorem}.}
In the bubbling convergence procedure, we have one additional assumption on the sequence $u^j$, that is: when restricted to $\partial D$, $u^j: \partial D\to \Gamma$ has degree 1. This is because the restrictions of both $\bar v_0$ and $\bar v_1$ on $\partial D$ are degree 1 parametrizations of $\Gamma$ by the assumption, and each $u^j$ will be a continuous deformation from $\bar v_0$ or $\bar v_1$; (note that eventually we will let $u^j=\gamma^j(\cdot, s_j)$). \vskip 1.5mm

Now we show that among all the disk bubbles, when restricted to $\partial D$, one of them, say $v_0$, will have degree 1, and all others $\{v_1, \cdots, v_m\}$ will have degree 0. This is a simple corollary of the classical Courant-Lebesgue Lemma.  Note that in the last step of the proof of Theorem \ref{T:Compactness theorem} (section \ref{SS:blowup analysis}), two disk bubbles separate along a sequence half annuli $A_j=B_{s_j}(y_j)\setminus B_{t_j}(y_j)$, and the proof therein shows that the energy along these annuli converges to zero, i.e. 
\[ \lim_{j\to\infty}\int_{A_j}\vert\nabla u^j\vert^2=0.\]
Recall that we can choose $s_j$ and $t_j$ such that $\lim s_j/t_j=\infty$. By the Courant-Lebesgue Lemma, we can find a radius $\tilde t_j\in (t_j/\lambda_j, 2t_j/\lambda_j)$, so that 
\[ \lim_{j\to\infty}\int_{\partial B_{\tilde t_j}(y_j) \cap D}\vert\nabla u^j\vert^2=0. \]
This implies that the distances between the images of $u^j$ of the two end points of $\partial B_{\tilde t_j}(y_j) \cap D$ (which lie in $\partial D$)  converges to 0. Therefore,  when restricted to $\partial D$, among disk bubbles $v_0, v_1$, one of them has degree 1 whereas the other has degree 0.  By iterating this argument along each boundary blowup process, we can show that only one of the disk bubble has degree 1, and all others have degree 0 when restricted to $\partial D$.

\vskip 0.15in
\noindent{\bf Some further discussions.}
One main goal of the min-max construction for the fixed boundary problem is to produce a third non-minimizing minimal disk spanning $\Gamma$, which is a direct generalization of the work of Morse-Tompkins \cite{Morse-Tompkins39} and Shiffman \cite{Shiffman39} to the Riemannian setting. There is one issue left open in our current result. In fact, it will be good if one can restrict the sweepouts to all those $\sigma\in\Om_f$, where $\sigma(\cdot, t): \partial D\to \Gamma$ is a monotone parametrization for each $t\in [0, 1]$; if this could be done, our proof will show that one of the disk bubble has monotone boundary parametrization, and all other disk bubbles must map the boundary $\partial D$ to a point on $\Gamma$, so that they are punctured harmonic spheres. Thus if we assume additionally that the ambient manifold $\mc N$ has non-positive curvature, then by the uniqueness of harmonic maps all of these (punctured) harmonic spheres must be constant, so the min-max solution we obtain in Theorem \ref{T:main2} must be a third non-minimizing minimal disk. Our current mollification process could possibly destroy the monotonicity property.  It will also be good to reduce the regularity of the boundary curve $\Gamma$ to rectifiable curves, but our theory need a nice Fermi neighborhood of $\Gamma$, which requires it to be smooth.

\section{Appendix}

For the definition of Fermi coordinate system near $\Gamma\subset \mc N$ we refer to \cite[Appendix A]{LZ16}. We will say a quantity $A$ is $\alpha$-close to quantity $B$ if we have $1-\alpha\leq\frac{A}{B}\leq 1+\alpha$.

\begin{lemma}
\label{L:Fermi equivalent to Rn}
There is a constant $\kappa$ only depending on the $\mc N$ and $\Gamma$ and the embedding $\mc N\to\mb R^N$, such that for any $x\in\Gamma$, the $\kappa$-neighborhood $U_\kappa$ of $x$ has local Fermi coordinate system, satisfying the following condition:
let $g^1$ be the metric of $\mc N$ under Fermi coordinate, $g^2$ be the standard Euclidean metric under the Fermi coordinate system, and $g^3$ be the standard Euclidean metric of $\mb R^N$, then there is a constant $\alpha$ depending on $\kappa$ such that:

\begin{enumerate}
\item for any vector $V\in T_p\mc N\subset T_p\mb R^N$ where $p\in U_\kappa$, $\Vert V\Vert_{g^k}$ is $\alpha$-close to $\Vert V\Vert_{g^l}$ for $k,l\in\{1,2,3\}$;
\item for any pair of points $p_1,p_2\in U_\kappa$, $\Vert p_1-p_2\Vert_{g^2}$ is $\alpha$-close to $\Vert p_1-p_2\Vert_{g^3}$.
\end{enumerate}
\end{lemma}

\begin{proof}
We only need to show $\kappa$ exists for any $x\in\Gamma$, then by a covering argument we can prove $\kappa$ exists globally. Fix $x\in \Gamma$. 

(1) By the definition of Fermi coordinates, we can choose the Fermi coordinates such that $g^{1}_{ij}(x)=g^2_{ij}(x)=\delta_{ij}$. Then if $\kappa$ is small enough, locally $(1-\alpha)id \leq(g^{1})^{-1}g^2\leq (1+\alpha)id$, so we get (1) for $k,l\in\{1,2\}$. Since $M\to\mb R^N$ is an isometric embedding, the lengths of $V$ measured by $g^1,g^3$ are close when $\kappa$ is small enough, hence (1) is true.\vskip 1.5mm

(2) First we show $\Vert p_1-p_2\Vert_{g^2}$ is $\alpha$-close to the distance between $p_1,p_2$ in $\mc N$. Note for any curve connecting $p_1$ with $p_2$ in $U_{\kappa}$, the length of the curve evaluated by $g^1$ and $g^2$ should be $\alpha$-close, since the lengths of the tangent vectors of the curve are $\alpha$-close. Since the length between $p_1,p_2$ is the shortest distance among all curves, we see that $\Vert p_1-p_2\Vert_{g^2}$ is $\alpha$-close to the distance between $p_1,p_2$ on $\mc N$ (Note we can choose $\kappa$ small such that $U_\kappa$ is convex). \vskip 1.5mm
 
Since at each point $x$ of $\mc N$, scaling up makes $\mc N$ converge to the tangent space at $x$, thus the distance for points closed to $x$ on $\mc N$ is equivalent to the distance in $\mb R^N$. Then apply this argument to all Fermi neighbourhood with small $\kappa$ we get that $\Vert p_1-p_2\Vert_{g^3}$ is $\alpha$-close to the distance between $p_1,p_2$ in $\mc N$. It is easy to see that $\Vert p_1-p_2\Vert_{g^3}\leq \dist_{\mc N}(p_1,p_2)$, thus $\dist_{\mc N}(p_1,p_2)$ is $\alpha$-close to $\Vert p_1-p_2\Vert_{g^3}$ for some $\alpha$. Then we conclude this lemma.
\end{proof}

We also need this Poincar\'e type inequality with partial zero boundary values.
\begin{lemma}
\label{L:Poincare inequality with partial zero boundary values}
Given a function $f:D^+_r\to\mb R$ with $f\vert_{\partial_r^A}=0$, there exists a constant $C=C(r)$ such that 
\[\int_{D^+_r}f^2\leq C\int_{D^+_r}\vert\nabla f\vert^2\]
\end{lemma}

\begin{proof}
We just extend $f$ to whole $D_r$ by letting $f(x,-y)=f(x,y)$ under Cartesian coordinate. Then $f$ is a function on $D_r$ vanishes on boundary, hence we can apply the classical Poincar\'e inequality. Note this reflection doubles $\int f^2$ and $\int \vert\nabla f\vert^2$ at the same time, so we get the desired Poincare inequality.
\end{proof}

\subsection{Bubble convergence implies varifold convergence}

The main goal of this section is to prove that bubble convergence implies varifold convergence (c.f. \cite[Proposition A.3]{CM08}). Let us recall some notions from geometric measure theory; (for more details see \cite[Section 1.3]{CM08}, further details see \cite{Simon-GMT}). Let $\pi:G_2\mc N\to\mc N$ be the Grassmannian bundle of $2$-planes over $\mc N$, and let us consider the pairs $(X,F)$, where $X$ is a compact surfaces (not necessarily connected) and $F:X\to G_2\mc N$ is a measurable map such that $f:=\pi\circ F$ is in $W^{1,2}(X,\mc N)$. We also use $J_f$ to denote the Jacobian of $f$. We say that a sequence $X_i=(X_i,F_i)$ with uniformly bounded areas \emph{varifold converges} to $(X,F)$ if for all $h\in C^0(G_2\mc N)$ we have
\[\int_{X_i}h\circ F_i J_{f_i}\to \int_X h\circ F J_f.\]
This is a kind of weak notion of convergence of measures. 
There exists a distance function $d_V$, the varifold distance, which induces this topology.\vskip 1.5mm

Here is one important example: a varifold induced by a map. Let $u:D\to\mc N$ be a $W^{1,2}$-map, then a pair $(X,F)$ induced by $u$ is constructed as follows: $X$ is just $D$, and $F:X\to G_2\mc N$ is given by sending $x$ to $du(T_xX)$. This is only defined on the measurable space where $J_u$ is nonzero, but we can extend it arbitrarily to all of $X$ since the corresponding Radon measure on $G_2\mc N$ (i.e. the \emph{varifold induced by $u$}) given by $h\to\int_X h\circ F J_u$ is independent of the extension.\vskip 1.5mm

\begin{prop}
If a sequence $v^j: D\to \mc N$ of $W^{1,2}$-maps bubble converges to a finite collection of smooth maps $\{u_0,\cdots,u_m\}$ such that either $u_i: (D, \partial D)\to (\mc N, \Gamma)$ is a harmonic disk with free boundary or $u_i: S^2\to \mc N$ is a harmonic sphere, and the energy identity holds, then this sequence also varifold converges to the varifold induced by $\{u_0,\cdots,u_m\}$.
\end{prop}

We want to emphasis the energy equality, which appears in the last part of Theorem \ref{T:Compactness theorem}, plays a key role in the proof.

\begin{proof}(see also \cite[proof of Proposition A.3]{CM08}). 
For each $v^j$ we let $V^j$ denote the corresponding map to $G_2\mc N$. Similarly for each $u_i$ let $U_i$ denote the corresponding map to $G_2\mc N$. By the construction of bubble convergence (in the proof of Theorem \ref{T:Compactness theorem}), we can choose $m+1$ sequences of domains $\Omega_0^j,\cdots,\Omega_m^j\subset D$ that are pairwise disjoint for each $j$, so that for each $i=0,\cdots,m$ applying $\Psi_{i,j}^{-1}$ (the inverse of the corresponding conformal dilation) to $\Omega_i^j$ gives a sequence of domains converging to either $D\backslash \mc S_i$ (if it is a disk bubble) or $S^2\backslash \mc S_i$ (if it is a sphere bubble), and they account for all the energy by the energy identity, i.e.
\[\lim_{j\to\infty}\int_{D\backslash(\cup_i)\Omega_{i}^j}\vert\nabla v^j\vert^2=0.\]
In order to show varifold convergence, we only need to show for any $h\in C^0(G_2\mc N)$ that
\begin{displaymath}
\begin{split}
\int_D h\circ U_i J_{u_i}&=\lim_{j\to\infty}\int_{\Omega_i^j}h\circ V^j J_{v^j}\\
&=\lim_{j\to\infty}\int_{\Psi_{i,j}^{-1}(\Omega_i^j)}h\circ V^j\circ \Psi_{i,j}J_{(v^j\circ \Psi_{i,j})},
\end{split}
\end{displaymath}
where the last equality is the change of variable formula.\vskip 1.5mm

Given $\eps>0$ and $i$, let $\Omega_\eps^i$ be the set where $J_{u_i}\geq\eps$, then we only need to show
\[\int_{\Omega_\eps^i}h\circ U_i J_{u_i}=\lim_{j\to\infty}\int_{\Psi_{i,j}^{-1}(\Omega_i^j)} h\circ V^j\circ \Psi_{i,j}J_{(v^j\circ \Psi_{i,j})}.\]
Note $J_{v^j\circ \Psi_{i,j}}\to J_{u_i}$ in $L^1$ since $v_i^j\to u_i$ in $W^{1,2}$, so the measure of $\{x\in\Omega_\eps^i:J_{v^j\circ \Psi_{i,j}}<\eps/2\}$ goes to zero; the $W^{1,2}$-convergence implies that for given $\delta>0$, the measure of $\{x\in\Omega_\eps^i:J_{v^j\circ \Psi_{i,j}}\geq\eps/2,\text{ and }  |V^j\circ \Psi_{i,j}-U_i|\geq\delta \}$ goes to zero. Then by the dominated convergence theorem we get the desired identity. Thus we conclude the varifold convergence.
\end{proof}

\bibliographystyle{plain}
\bibliography{FB}
\end{document}